\pdfoutput=1
\documentclass[12pt,reqno]{amsart}
\usepackage[lmargin=1.25in,rmargin=1.25in,tmargin=1in,bmargin=1in]{geometry}

\usepackage{yfonts}
\usepackage{amsthm,amsfonts,amssymb}

\usepackage{mathrsfs}
\usepackage[scaled=1.15,mathscr]{urwchancal}

\usepackage{enumitem}
\usepackage[colorlinks]{hyperref}    
\hypersetup{linkcolor=blue, citecolor=blue}

\renewcommand{\epsilon}{\varepsilon}

\newcommand{\tops}[2]{\texorpdfstring{#1}{#2}}
\newcommand{\qqand}{\qquad\text{and}\qquad}

\def\Xint#1{\mathchoice
   {\XXint\displaystyle\textstyle{#1}}%
   {\XXint\textstyle\scriptstyle{#1}}%
   {\XXint\scriptstyle\scriptscriptstyle{#1}}%
   {\XXint\scriptscriptstyle\scriptscriptstyle{#1}}%
   \!\int}
\def\XXint#1#2#3{{\setbox0=\hbox{$#1{#2#3}{\int}$}
     \vcenter{\hbox{$#2#3$}}\kern-.5\wd0}}
\def\dint{\Xint-}

\renewcommand{\Re}{\operatorname{Re}}

\let\originalleft\left
\let\originalright\right
\renewcommand{\left}{\mathopen{}\mathclose\bgroup\originalleft}
\renewcommand{\right}{\aftergroup\egroup\originalright}

\newcommand{\pfr}[2]{\left(\frac{#1}{#2}\right)}
\newcommand{\fr}[2]{\frac{#1}{#2}}
\newcommand{\nfr}[2]{#1/#2}
\newcommand{\tf}[2]{\tfrac{#1}{#2}}

\newtheorem{theorem}{Theorem}
\newtheorem{proposition}[theorem]{Proposition}
\newtheorem{lemma}[theorem]{Lemma}
\newtheorem{corollary}[theorem]{Corollary}

\theoremstyle{definition}
\newtheorem{remark}[theorem]{Remark}
\newtheorem{definition}[theorem]{Definition}

\numberwithin{theorem}{section}
\numberwithin{equation}{section}
\numberwithin{figure}{section}

\newcommand{\<}{\begin{equation}}
\renewcommand{\>}{\end{equation}}

\newcommand{\expp}[1]{\exp\left( #1 \right)}

\newcommand{\Dir}[2]{\mathfrak{D}[#1;#2]}

\newcommand{\floor}[1]{\lfloor#1\rfloor}

\newcommand{\brk}[1]{\left[#1\right]}
\renewcommand{\O}[1]{O\left(#1\right)}

\newcommand{\ldint}[1]{\dint_{#1-i\infty}^{#1+i\infty}}

\newcommand{\less}{\ll}
\newcommand{\great}{\gg}
\newcommand{\mop}[1]{(#1 \partial_{#1})}
\newcommand{\rt}[1]{\sqrt{#1}}

\let\mod\undefined
\newcommand{\mod}[1]{\,(#1)}

\renewcommand{\;}{\hspace{0.04em}} 
\renewcommand{\:}{\hspace{0.08em}} 
\renewcommand{\.}{\hspace{-0.08em}}

\newcommand{\Vmaj}{V\hspace{-0.1em}} 
\newcommand{\Wmaj}{W\hspace{-0.08em}} 

\newcommand{\PsiD}[1]{\Psi_{(#1)}}
\newcommand{\Psid}[1]{\Psi^{(#1)}}

\newcommand{\logX}{\log X}
\newcommand{\logx}{\log x}

\newcommand{\lf}{\left}
\newcommand{\lh}{\left}
\newcommand{\rh}{\right}

\newcommand{\signs}{\{0,\pm1\}}
\newcommand{\spnf}{(p(n,f))_\nn}
\newcommand{\spnA}{(p(n,A))_\nn}

\newcommand{\keep}[1]{{\binoppenalty=10000\relpenalty=10000 #1}}

\newcommand{\cc}{\mathbb{C}}
\newcommand{\nn}{\mathbb{N}}
\newcommand{\pp}{\mathbb{P}}
\newcommand{\rr}{\mathbb{R}}
\newcommand{\zz}{\mathbb{Z}}

\newcommand{\fm}{\mathfrak{m}}
\newcommand{\fz}{\mathfrak{z}}

\newcommand{\fA}{\mathscr{A}}
\newcommand{\fD}{\mathfrak{D}}
\newcommand{\fM}{\mathfrak{M}}
\newcommand{\fP}{\mathfrak{P}}
\newcommand{\fT}{\mathfrak{T}}

\newcommand{\xa}{\alpha}
\newcommand{\xb}{\beta}
\newcommand{\xd}{\delta}
\newcommand{\xe}{\epsilon}

\newcommand{\xh}{\eta}
\newcommand{\xk}{\kappa}
\newcommand{\xl}{\lambda}
\newcommand{\xm}{\mu}
\newcommand{\xn}{\nu}
\newcommand{\xq}{\theta}
\newcommand{\xr}{\rho}
\newcommand{\xs}{\sigma}
\newcommand{\xt}{\tau}
\newcommand{\vphi}{\varphi}
\newcommand{\xz}{\zeta}

\newcommand{\xG}{\Gamma}
\newcommand{\xD}{\Delta}
\newcommand{\xQ}{\Theta}
\newcommand{\xO}{\Omega}

\newcommand{\cQ}{\mathcal{Q}}
\newcommand{\cX}{\mathcal{X}}

\usepackage{fancyhdr}
\usepackage{tikz,caption}
\usetikzlibrary{calc,decorations.markings}

\def\centerarc[#1](#2)(#3:#4:#5)
    { \draw[#1] ($(#2)+({#5*cos(#3)},{#5*sin(#3)})$) arc (#3:#4:#5); }

 \usepackage[foot]{amsaddr}

 \makeatletter
 \renewcommand{\email}[2][]{%
   \ifx\emails\@empty\relax\else{\g@addto@macro\emails{,\space}}\fi%
   \@ifnotempty{#1}{\g@addto@macro\emails{\textrm{(#1)}\space}}%
   \g@addto@macro\emails{#2}%
 }
 \makeatother

\usepackage[nobysame,initials]{amsrefs}

\AtBeginDocument{%
   \def\MR#1{}
}

\makeatletter
\def\section{%
    \@startsection{section}{1}%
    \z@{.7\linespacing\@plus\linespacing}{.5\linespacing}%
    {\normalfont\large\bfseries}%
}
\makeatother

\makeatletter
\def\@seccntformat#1{%
  \protect\textup{\protect\@secnumfont
    \csname the#1\endcsname
\space\space
  }%
}
\makeatother

\begin{document}

\title{Bounds on the M\"obius-signed partition numbers} 
\author[T.\ Daniels]{Taylor Daniels}
\address{Purdue University, West Lafayette IN}
\email{daniel84@purdue.edu}
\begin{abstract}
    For $n \in \mathbb{N}$ let $\Pi[n]$ denote the set of partitions of $n$, i.e., the set of positive integer tuples $(x_1,x_2,\ldots,x_k)$ such that \keep{$x_1 \geq x_2 \geq \cdots \geq x_k$} and \keep{$x_1 + x_2 + \cdots + x_k = n$}. 
    Fixing $f:\mathbb{N}\to\{0,\pm 1\}$, for $\pi = (x_1,x_2,\ldots,x_k) \in \Pi[n]$ let $f(\pi) := f(x_1)f(x_2)\cdots f(x_k)$. In this way we define the {signed partition numbers} 
        \[ p(n,f) = \sum_{\pi\in\Pi[n]} f(\pi). \]

    Following work of Vaughan and Gafni on partitions into primes and prime powers, we derive asymptotic formulae for $p(n,\mu)$ and $p(n,\lambda)$, where $\mu$ and $\lambda$ denote the M\"obius and Liouville functions from prime number theory, respectively. In addition we discuss how quantities $p(n,f)$ generalize the classical notion of restricted partitions.
\end{abstract}
\maketitle

\section{Introduction}

For $n \in \nn$ let $\Pi[n]$ denote the set of \emph{partitions} of $n$, i.e., the set of positive integer tuples $(x_1,x_2,\ldots,x_k)$ such that \keep{$x_1 \geq x_2 \geq \ldots \geq x_k$} and \keep{$x_1 + x_2 + \cdots + x_k = n$}. 
For fixed $f:\nn\to\signs$, for $n \in \nn$ and $\pi=(x_1,x_2,\ldots,x_k) \in \Pi[n]$ let
    \<
        \label{eq:fpi}
        f(\pi) := f(x_1)f(x_2)\cdots f(x_k).
    \>
With this we define the \emph{signed partition numbers} $p(n,f)$ via
    \<
        \label{eq:pnfDefin}
        p(n,f) = \sum_{\pi\in\Pi[n]} f(\pi).  
    \>

Definition \eqref{eq:pnfDefin} generalizes several classical partition-related quantities. For example, with the constant function $1$ the $p(n,1)$ are the \emph{ordinary partition numbers} $p(n)$, and with the indicator function $1_A$ for some $A \subset \nn$ the $p(n,1_A)$ are the \emph{$A$-restricted partition numbers} $p_A(n)$, which we write as $p(n,A)$. 

Many sequences of restricted partition numbers are classical and well studied (see, e.g., \cites{andrews1976partitions,erdos1942elementary}). Despite this, the past ten years have seen a number of new publications on the asymptotics of some restricted partition numbers, e.g., \cites{berndt2018partitions,das2022partitions,gafni2016power,vaughan2015squares}.  

Our two sequences $\spnf$ of interest are those using (for $f$) the M\"obius $\mu$ and Liouville $\xl$ functions from multiplicative number theory. If $n \in \nn$ has prime factorization $p_1^{a_1}p_2^{a_2} \cdots p_r^{a_r}$ with distinct primes $p_i$ and all $a_i \geq 1$, then
    \[
        \xl(n) = (-1)^{a_1 + \cdots + a_r} \qquad \text{and} \qquad \mu(n) = 
            \begin{cases}
                (-1)^r & \text{if all $a_i = 1$,} \\
                0 & \text{otherwise.}
            \end{cases} 
    \]
We note that $\mu(n) = \xl(n)$ on the support of $\mu$ and $\mu^2$, i.e., on the set of \emph{squarefree} $n$.

We first state simple asymptotic results on the quantities $p(n,\mu)$ and $p(n,\xl)$ which are immediate corollaries of our main results. 

\begin{theorem}
    \label{thm:BigOh}
    For all $\xe > 0$, as $n\to\infty$ one has
        \begin{align}
            \label{eq:bigOh}
            p(n,\mu) = O\big(e^{(1+\xe)\sqrt{n}}\big) \qqand p(n,\xl) = O\big(e^{(1+\xe)\sqrt{\xz(2)n}}\big),
        \end{align}
    where $\xz(2) = \pi^2/6$. In addition, for $n$ even, say $n=2k$, as $k\to\infty$ one has
        \<
            \label{eq:logpn}
            \log p(2k,\mu) \sim \sqrt{2k} \qqand \log p(2k,\xl) \sim \sqrt{\zeta(2)2k},
        \>
    where the relation $a_m \sim b_m$ indicates that $\lim_{m\to\infty} a_m/b_m = 1$.
\end{theorem}

The generating function $\Phi(z) = \Phi(z,f)$ for the sequence $\spnf$ is defined via
    \<
        \label{eq:Phi}
    	\Phi(z) = \prod_{n=1}^\infty (1-f(n)z^n)^{-1} = 1 + \sum_{n=1}^\infty p(n,f)z^n \qquad (|z|<1),
    \>
and $\Psi(z) = \Psi(z,f)$ is defined via
    \<
        \label{eq:PsiDefin}
        \Psi(z) = \sum_{k=1}^\infty \sum_{n=1}^\infty \frac{f^k(n)}{k} z^{nk} \qquad (|z| < 1),
    \>
so that 
    \<
        \Phi(z) = \exp \Psi(z).
    \>
By comparison to the functions $\Phi(z,1)$ and $\Psi(z,1)$, one sees that $\Phi(z,f)$ and $\Psi(z,f)$ are analytic for $|z|<1$, so by Cauchy's theorem one has
	\<
		\label{eq:pnfInt}
		p(n,f) = \frac{1}{2\pi i}\int_{|z|=\rho} \Phi(z)z^{-n-1} \,dz = \rho^{-n} \int_0^1 \Phi(\rho e(\xa)) e(-n\xa) \,d\xa
	\>
for all $0<\rho<1$, where $e(\xa) := \exp(2\pi i\xa)$.  Using integrals \eqref{eq:pnfInt}, we define the quantities $p(x,f)$ for general $x>0$ by using $x$ in place of $n$ therein.

When $f$ is nonnegative and $f(1)=1$, the integrals \eqref{eq:pnfInt} are dominated by the effect of the singularity %
\footnote{Strictly speaking, the analytic function $\Phi(z)$ may not have singularities in the classical sense, since the ``singularities'' of $\Phi(z)$ are likely to be dense in the unit circle and therefore not isolated. This technicality does not impact any of our rigorous arguments, and thus, lacking a better descriptor, we simply discuss the ``singularities'' of $\Phi(z)$ on the unit circle.}%
of $\Phi(z)$ at $z = +1$. Thus, for a given $x$ it is natural to select $\rho=\rho(x)$ in \eqref{eq:pnfInt} so as to minimize the quantity $\rho^{-x}\Phi(\rho)$, which is equivalent to minimizing $\Psi(\rho)-x\log\rho$. For these nonnegative $f$, the quantity $\Psi(\rho)-x\log\rho$ grows to $+\infty$ as $\rho$ approaches $0$ and $1$, and it follows that the minimum of said quantity occurs when
    \<
        \label{eq:spNonNeg}
        \rho \Psi'(\rho) = x.    
    \>
Moreover, such a $\rho$ satisfying \eqref{eq:spNonNeg} is necessarily unique, as the nonnegativity of $f$ implies that $\rho \Psi'(\rho)$ is strictly increasing on $(0,1)$. Equation \eqref{eq:spNonNeg} is the \emph{saddle-point equation} common to many analytic-combinatorial problems. 

When $f = \mu$ or $f = \xl$, the integrals \eqref{eq:pnfInt} for $p(x,\mu)$ and $p(x,\xl)$ are both dominated by the effects of \emph{two} singularities, namely those at $z = +1$ and $z = -1$. We are thus motivated to (individually) minimize the quantities $\Psi(\rho) - x \log \rho$ and $\Psi(-\rho) - x \log \rho$, leading to the equations
    \begin{align}
        \label{eq:sdp}
        \rho \Psi'(\rho) &= x, \tag{sp} \\
        \label{eq:sdp*}
        (-\rho)\Psi'(-\rho) &= x, \tag{sp$*$}
    \end{align}
where $(-\xr)\Psi'(-\xr)$ is $z\Psi'(z)$ evaluated at $z=-\rho$. 
Because $\mu$ and $\xl$ take both positive and negative values, it is no longer clear that solutions to equations \eqref{eq:sdp} and \eqref{eq:sdp*} even exist, much less unique ones. This is in fact the case for all sufficiently large $x$, as proved in section \ref{sec:Relations}.

For integer $j \geq 0$ let $\PsiD{j}(z) := \mop{z}^j\Psi(z)$. To preempt any confusion, we note that $(-1)^{j}\PsiD{j}(-\rho)$ and $\PsiD{j}(-\rho)$ are uniquely valued, regardless of their interpretations vis-\`a-vis the ``chain rule''.
When $x$ is sufficiently large and $\rho\in(0,1)$ is a solution to one of the equations \eqref{eq:sdp} and \eqref{eq:sdp*} with $\Psi(z)=\Psi(z,\mu)$, we have the following relations between $x$, $\rho$, and certain derivatives of $\Psi(\pm\rho)$.

\begin{theorem}
    \label{M:thm:Relations}
    Let $A > 0$ and $\Psi(z)=\Psi(z,\mu)$. For all sufficiently large $x$, if $\rho = \rho(x)$ satisfies at least one of the equations \eqref{eq:sdp} and \eqref{eq:sdp*}, then
        \begin{align*}
            - x \log\rho &= \tf12 \sqrt{x} \lf[ 1+O((\logx)^{-A}) \rh], \\
            \Psi(\pm\rho) &= \tf12 \sqrt{x} \lf[ 1+O((\logx)^{-A}) \rh].
        \end{align*}
    In general, for these $x$ and $\rho$ and for $j \geq 1$, one has
        \begin{align*}
            (\pm 1)^j\Psid{j}(\pm\rho) &= j!\:2^{j-1}x^{\fr{j+1}{2}} \lf[ 1+O_j((\logx)^{-A}) \rh], \\
            \PsiD{j}(\pm\rho) &= j!\:2^{j-1}x^{\fr{j+1}2} \lf[ 1+O_j((\logx)^{-A}) \rh].
        \end{align*}
\end{theorem}

Our next theorem gives a formula for $p(x,\mu)$ modeled after Vaughan's asymptotic formula for $p(n,\pp)$, the number of partitions into prime numbers, as in \cite{vaughan2008number}*{Thm.\ 2}. In particular, with $\Phi(z)$ and $\Psi(z)$ denoting $\Phi(z,\pp)$ and $\Psi(z,\pp)$, respectively, Vaughan establishes that as $n\to\infty$ one has
    \<
        \label{eq:pnP}
        p(n,\pp) = \frac{\rho^{-n}\Phi(\rho)}{\rt{2\pi\PsiD{2}(\rho)}} \big[ 1+O(n^{-\nfr15}) \big],
    \>
where $\rho \in (0,1)$ is the unique solution to the equation $\rho \Psi'(\rho) = n$. 

\begin{theorem}
    \label{M:thm:asymp}
    Let $\Phi(z)=\Phi(z,\xm)$ and $\Psi(z)=\Psi(z,\xm)$, and for sufficiently large $x$ let $\rho$ and $\rho_*$ denote the unique solutions to \eqref{eq:sdp} and \eqref{eq:sdp*}, respectively. As $x \to \infty$, one has
        \<
            \label{M:eq:pxmuForm}
            p(x,\mu) = \frac{\rho^{-x}\Phi(\rho)}{\sqrt{2\pi \PsiD{2}(\rho)}} \big[ 1+O(x^{-\nfr15}) \big] + (-1)^{-x} \frac{\rho_*^{-x}\Phi(-\rho_*)}{\sqrt{2\pi \PsiD{2}(-\rho_*)}} \big[ 1+O(x^{-\nfr15}) \big],
        \>
    where $(-1)^{-x} = \exp(-\pi ix)$.
\end{theorem}

For the case of $\Phi(z,\xl)$ and $\Psi(z,\xl)$ we have results similar to those above. Due to its ubiquity in our discussions of $\Psi(z,\xl)$, it is convenient to let
    \<
        \fz = \xz(2) = \tf{\pi^2}{6}.
    \>

\begin{theorem}
	\label{L:thm:Relations}
	Let $A>0$ and $\Psi(z)=\Psi(z,\xl)$. For all sufficiently large $x$, if $\rho = \rho(x)$ satisfies at least one of the equations \eqref{eq:sdp} and \eqref{eq:sdp*}, then
		\begin{align*}
			- x \log\rho &= \tf12\rt{\fz x} \brk{ 1+O((\logx)^{-A}) }, \\
            \Psi(\pm\rho) &= \tf12 \rt{\fz x} \brk{ 1+O((\logx)^{-A}) }.
        \end{align*}
    In general, for these $x$ and $\rho$ and for $j \geq 1$, one has
        \begin{align*}
            (\pm1)^{j}\Psid{j}(\pm\rho) &=  j!\: 2^{j-1} \fz^{\frac{1-j}{2}} x^{\fr{j+1}2} \brk{ 1+O_j((\logx)^{-A}) }, \\
            \PsiD{j}(\pm\rho) &= j!\: 2^{j-1} \fz^{\frac{1-j}{2}} x^{\fr{j+1}2} \brk{ 1+O_j((\logx)^{-A}) },
        \end{align*}
\end{theorem}

\begin{theorem}
    \label{L:thm:pnlAsymp}
    Let $\Phi(z)=\Phi(z,\xl)$ and $\Psi(z)=\Psi(z,\xl)$, and for sufficiently large $x$ let $\rho$ and $\rho_*$ denote the unique solutions to \eqref{eq:sdp} and \eqref{eq:sdp*}, respectively. As $x\to\infty$ one has
        \<
            \label{L:eq:pxxlForm}
            p(x,\xl) = \frac{\rho^{-x}\Phi(\rho)}{\rt{2\pi \PsiD{2}(\rho)}} \big[ 1+O(x^{-\nfr15}) \big] + (-1)^{-x} \frac{\rho_*^{-x}\Phi(-\rho_*)}{\rt{2\pi \PsiD{2}(-\rho_*)}} \big[ 1+O(x^{-\nfr15}) \big],
        \>
    where $(-1)^{-x} = \exp(-\pi ix)$.
\end{theorem}

We note that the results of Theorem \ref{thm:BigOh} immediately follow from Theorems \ref{M:thm:Relations}--\ref{L:thm:pnlAsymp}. Indeed, for the case of $p(n,\mu)$, using the relations in Theorem \ref{M:thm:Relations} one may deduce that the two fractions in equation \eqref{M:eq:pxmuForm} both grow like $\exp(\sqrt{n} + o(\sqrt{n}))$, from which the bound \eqref{eq:bigOh} for $p(n,\mu)$ follows at once. Because $(-1)^{-x} = +1$ when $x=2k$ with $k\in\nn$, the asymptotic relation for $\log p(2k,\mu)$ in \eqref{eq:logpn} similarly follows at once. The statements concerning $p(n,\xl)$ and $\log p(2k,\xl)$ similarly follow from Theorems \ref{L:thm:Relations} and \ref{L:thm:pnlAsymp}.

We now discuss three classical results on certain sequences $\spnA$ to help contextualize the relations of Theorem \ref{thm:BigOh}. The most well known of these is the relation 
    \[
        \log p(n,1) \sim \xk \rt{n}, 
    \]
where
    \[
        \xk := \pi \sqrt{2/3},    
    \]
which is an immediate corollary of Hardy and Ramanujan's main results in \cite{hardy1918asymptotic}.

Before stating the second result, we recall that $A \subset \nn$ is said to have \emph{density} $\xd_A$ if the ratio $|A \cap \{1,2,\ldots,N\}|N^{-1}$ tends to $\xd_A$ as $N\to\infty$. It is evident that $\xd_\nn = 1$, and it is well known that the set of squarefree numbers has density $6/\pi^2 \approx 0.601$. The prime number theorem implies that $\xd_\pp = 0$. A remarkable theorem of Erd\H{o}s \cite{erdos1942elementary} states that: if $A \subset \nn$ and $\mathrm{gcd}(A)=1$, then $A$ has density $\xd_A > 0$ if and only if
    \<
        \label{eq:Erdossim}
        \log p(n,A) \sim \xk \sqrt{\xd_A n}.
    \>
We remark that the proof of the ``only if'' statement in Erd\H{o}s' theorem is omitted in \cite{erdos1942elementary} since the proof of a similar result is given there, and that a complete proof of both statements may be found in \cite{nathanson2008elementary}*{Ch.\ 16}. The third result, though originally due to Roth and Szekeres \cite{roth1954some}, can be deduced as a corollary to Vaughan's formula \eqref{eq:pnP}. Namely, as $n\to\infty$ one has
	\<
		\label{eq:LogPrimessim}
		\log p(n,\pp) \sim \xk\rt{\fr{2n}{\log n}}.
	\>

Noting that $\xk = 2\rt{\xz(2)} = 2\rt{\fz}$, Theorem \ref{thm:BigOh} states that as $n\to\infty$ one has
    \[
        p(n,\mu) = O(e^{(1+\xe)\sqrt{n}}) \qqand p(n,\xl) = O(e^{(\fr12+\xe)\xk\sqrt{n}}),   
    \]
and that for $n = 2k$ one has
    \[
        \log p(2k,\mu) \sim \sqrt{2k} \qqand \log p(2k,\xl) \sim \tf12 \xk \sqrt{2k}.    
    \]

Finally we comment on our use of the term ``signed partition number'', since it is apparent from \eqref{eq:fpi} and \eqref{eq:pnfDefin} that ``weighted partition number'' is an equally appropriate term for $p(n,f)$. We avoid this latter term due to its ubiquity in the literature as a moniker for sequences having generating functions of the form
    \<
        \label{eq:weightedp}
        \Phi(z) = \prod_{n=1}^\infty (1-z^n)^{-f(n)} = 1 + \sum_{n=1}^\infty p_f(n) z^n    
    \>
for some nonnegative integer-valued $f(n)$. These $p_f(n)$ coincide with our $p(n,f)$ if $f$ takes only the values $0$ and $1$, but given the different roles $f(n)$ plays in \eqref{eq:Phi} and \eqref{eq:weightedp}, we favor the descriptor ``signed'' for $p(n,f)$ rather than ``weighted.''

\subsection*{Acknowledgements}

The author would like thank Trevor Wooley for both his suggestion of these research problems and his financial support, and would like to additionally thank both Trevor Wooley and Ben McReynolds for their numerous, invaluable suggestions and feedback in preparing this work. Thanks are also extended to Nicolas Robles and Alexandru Zaharescu for many helpful conversations in formulating this paper.

\section{Notation and structure}
\label{sec:prelim}

Expressions $O(g(x))$ with $g(x) > 0$ are quantities bounded by some constant multiple of $g(x)$. The relation $f(x) \less g(x)$ indicates that $f(x) = O(g(x))$ for all sufficiently large $x>0$, and $f(x)\asymp g(x)$ indicates that $f(x) \ll g(x)$ and $g(x) \ll f(x)$. Quantities of the form $X_0(A)$ indicate some constant $X_0 > 0$ depending on $A$. Statements involving $\xe$ are understood to hold for all sufficiently small $\xe > 0$ unless noted otherwise, and constants' dependencies on $\xe$ are generally suppressed.  

For real $\xa$ let $e(\xa)=\exp(2\pi i\xa)$, let $\|\xa\| = \min_{n\in\zz}|\xa-n|$, and let $\{\xa\} = \xa - \floor{\xa}$, and for real $\xa$ and $f:\nn\to\cc$ let
    \[
        S_f(t,\xa) = \sum_{n \leq t} f(n)e(n\xa).
    \]
Expressions $n \equiv a$ (mod $q$) are abbreviated as $n \equiv a \mod{q}$. We recall $\PsiD{j}(z) := \mop{z}^j \Psi(z)$. 

Let $1 \leq Q \leq X$. If $0\leq a \leq q$ and $(a,q)=1$, we define the \emph{arc} about $a/q$ as the set
	\[
		\fM(a/q) = \lf\{ \xa \in [0,1) : |\xa - a/q| \leq Q/(qX) \rh\},
	\]
noting that these sets are generally denoted $\fM(q,a)$ in the literature. Among the sets $\fM(a/q)$, the \emph{major arcs} are those $\fM(\nfr{a}{q})$ with $q \leq Q$, and $\fM(X,Q)$ denotes the union of all major arcs. The connected components of $[0,1) \setminus \fM(X,Q)$ are the \emph{minor arcs}, and their union is denoted $\fm(X,Q)$.

Ultimately, only $\xa$ lying in small subsets of the arcs $\fM(0/1)$, $\fM(1/2)$, and $\fM(1/1)$ significantly contribute to the integral \eqref{eq:pnfInt}, with all other $\xa$ contributing to an error term. Anticipating this, for $X,X_*>1$ we define
    \begin{align*}
        \fP &= \{ \xa \in [0,1) : \|\xa\| \leq X^{-1}(\logX)^{-\fr14} \}, \\
        \fP_* &= \{\xa \in [0,1) : |\xa - \tf12| \leq X_*^{-1}(\logX_*)^{-\fr14} \}.
    \end{align*}
Following Gafni \cite{gafni2021partitions} we term the sets $\fP$ and $\fP_*$ the \emph{principal arcs}, though colloquially the terms ``super-major arcs'' and ``major-major arcs'' have also been used. 

Functions $\Phi(z,f)$, $\Psi(z,f)$, etc.\ are simply written as $\Phi(z)$, $\Psi(z)$, etc. when the $f$ used is clear or when the discussions apply to general $f:\nn\to\signs$.
For $\rho, \rho_* \in (0,1)$ it is convenient to write
    \<
        \label{eq:X}
        X = \frac{1}{\log(1/\rho)} \qqand X_* = \frac{1}{\log(1/\rho_*)},
    \>
so that
    \<
        \label{eq:rho}
        \rho = e^{-1/X} \qqand \xr_* = e^{-1/X_*},
    \>
and this convention is maintained throughout this paper.

\subsection*{Structure}
This paper has three major components. Section \ref{sec:template} provides a general recipe for our analyses, after which sections \ref{sec:MinorArcs}--\ref{sec:PrincArcsII} consider the quantities $\Psi(\xr e(\xa),\mu)$ when $\xa$ is in the minor, major, and principal arcs and $\xr$ is sufficiently close to $1$ (rather, when $X$ is sufficiently large). In section \ref{sec:Relations} we examine the implicit relationships between $\xr$, $X$, and $x$ when $x$ is large and $\xr=\xr(x)$ is defined as the (unique) solution to equation \eqref{eq:sdp}. Similar analyses are made for $\xr_*$ and $X_*$ with $\xr_*$ defined as the solution to \eqref{eq:sdp*}, concluding the proof of Theorem \ref{M:thm:Relations}. 

Using our results on $\Psi(z,\mu)$, sections \ref{sec:transference} and \ref{sec:nonPrc} establish that only those $\xa$ in the principal arcs significantly contribute to the integral \eqref{eq:pnfInt} for $p(x,\mu)$. Following this, section \ref{sec:MAsymp} treats these principal-arc integrals to ultimately establish Theorem \ref{M:thm:asymp}.

Finally, section \ref{sec:Liouville} considers the quantities $p(x,\xl)$. Because the Liouville $\xl$ function is far easier to handle than the M\"obius $\mu$ function, the techniques of sections \ref{sec:template}--\ref{sec:Relations} allow for a rapid analysis of $\Psi(z,\xl)$ and proof of Theorem \ref{L:thm:Relations}. Due to the overwhelming similarities of the relations of Theorems \ref{M:thm:Relations} and \ref{L:thm:Relations}, the arguments of sections \ref{sec:transference}--\ref{sec:MAsymp} that prove Theorem \ref{M:thm:asymp} allow for an almost-immediate proof of Theorem \ref{L:thm:pnlAsymp}.

\subsection*{A note on \tops{$\xr_*$}{ρ*} and \tops{$X_*$}{X*}}

In sections \ref{sec:template}--\ref{sec:PrincArcsII} our results involve only $\xr$ and $X$, and not $\xr_*$ and $X_*$. Indeed, until section \ref{sec:Relations} our results do not require that $\rho$ (or $\rho_*$) be a solution to equation \eqref{eq:sdp} (or \eqref{eq:sdp*}), so there is no need to discuss two different $\xr,\xr_* \in (0,1)$ until then. Moreover, in section \ref{sec:Relations} we find that  when $\xr$ and $\xr_*$ are the (unique) solutions to \eqref{eq:sdp} and \eqref{eq:sdp*}, respectively, one has $X \sim X_* \sim 2\sqrt{x}$, making the results of sections \ref{sec:template}--\ref{sec:PrincArcsII} equally applicable if one uses $\xr_*$ and $X_*$ in place of $\xr$ and $X$, respectively.

\section{A template for our analyses}
\label{sec:template}

Our analyses of $\Psi(z,\mu)$ and $\Psi(z,\xl)$ for $z = \rho e(\xa)$ with $\xa$ in the sets $\fM(X,Q)$ and $\fm(X,Q)$ more or less follow the same sequence of steps. This section outlines these steps and summarizes them in Lemma \ref{F:lem:Fundamental}.

By comparison with $\Psi(z,1)$, the double series for $\Psi(z,f)$ in equation \eqref{eq:PsiDefin} converges absolutely for $|z|<1$. Thus, we exchange the sums in \eqref{eq:PsiDefin} and define
    \[
        F_k(z) = \sum_{n=1}^\infty f^k(n) z^n
    \]
so that
    \[
        \Psi(z) = \sum_{k=1}^\infty \sum_{n=1}^\infty \frac{f^k(n)}{k} z^{nk} = \sum_{k=1}^\infty \fr{1}{k} F_k(z^k).    
    \]

A cornerstone of our analyses is the observation that if $f(n)$ only takes values $0$ or $\pm 1$, then the functions $F_k(z)$ are only unique up to the parity of $k$. Indeed, as $f^k(n) = f(n)$ for odd $k$ and $f^k(n) = f^2(n) = |f(n)|$ for even $k$, we see that $F_k(z) = F_{k \,\,(\mathrm{mod}\, 2)}(z)$ for all $k$. We thus let
    \begin{alignat}{2}
        \label{F:eq:F0F1Defin}
        F_0(z) &= \sum_{n=1}^\infty f^2(n)z^n, \qquad && F_1(z) = \sum_{n=1}^\infty f(n)z^n, \\
        \label{F:eq:Psi0Psi1Defin}
        \Psi_0(z) &= \sum_{k \equiv 0 \mod2} \frac{1}{k} F_0(z^k), \qquad && \Psi_1(z) = \sum_{k\equiv 1 \mod2} \frac{1}{k} F_1(z^k),
    \end{alignat}
so that
    \[
        \Psi(z) = \Psi_0(z) + \Psi_1(z).
    \]

\begin{remark}
    We write $F_0(z)$ instead of $F_2(z)$ because the sums $\sum_n f^{2}(n)z^n$ provide the main asymptotic terms in our analyses of $p(n,\mu)$ and $p(n,\xl)$. 
\end{remark}

For $j \in \{0,1\}$ we split the functions $\Psi_j(z)$ into leading sums and tail sums, namely
    \<
        \label{F:eq:PsiHeadTailSums}
        \Psi_j(z) = \sum_{\substack{k \leq K\\k\equiv j \mod2}} \fr{1}{k} F_j(z^k) + \sum_{\substack{k > K\\ k \equiv j \mod2}} \fr{1}{k} F_j(z^k), 
    \>
where $K > 1$ is a real parameter of our choosing. The tail sums in equation \eqref{F:eq:PsiHeadTailSums} are always part of our error terms, so we record the following lemma bounding these tail sums independently of $j$, and even independently of $f$.

\begin{lemma}
    \label{F:lem:TailBB}
    Let $f : \nn\to\signs$. For real $\xa$ and large $X$ and $K$ one has that
        \[
            \sum_{k > K} \frac{1}{k} F_k(\rho^k e(k\xa)) \less \frac{X}{K}.
        \]
\end{lemma}

\begin{proof}
    For fixed $k$ we see that
        \[
            F_k(\rho^k e(k\xa)) = \sum_{n=1}^\infty f^k(n) \rho^{nk} e(nk \xa) \less \sum_{n=1}^\infty \rho^{nk} = \frac{\rho^k}{1 - \rho^k}.    
        \]
    Using the relation $\rho = e^{-1/X}$, it follows that
        \[
            \sum_{k > K} \frac{1}{k} F_k(\rho^k e(k\xa)) \less \sum_{k > K} \frac{1}{k} \, \frac{ e^{-k/X} }{1 - e^{-k/X} } = \sum_{k > K} \frac{1/k}{ e^{k/X} - 1 } \less \sum_{k > K} \frac{X}{k^2},    
        \]
    and the result follows.
\end{proof}

\begin{remark}
    It is clear that Lemma \ref{F:lem:TailBB} extends to all bounded $f : \nn \to \cc$.
\end{remark}

The analyses of $\Psi_0(z)$ and $\Psi_1(z)$ are thus reduced to analyses of $F_0(z^k)$ and $F_1(z^k)$ for $k \leq K$. In the following discussion let $j \in \{0,1\}$ and let $X$ be sufficiently large. 
Suppose that for $z$ in some arbitrary $\Omega \subset \{|z|<1\}$  there exist functions $\cX_0(z)$ and $\cX_1(z)$ such that: For $z \in \xO$ and $k \leq K$ one has
    \<
        \label{F:eq:FjCalXj}
        F_j(z^k) = \cX_j(z)/k.
    \>
Then
    \[
        \sum_{\substack{k \leq K \\ k \equiv j\mod2}} \frac{\cX_j(z)}{k^2} = c_j \cX_j(z) + O\pfr{\cX_j(z)}{K} \qquad (z \in \xO),
    \]
where $c_0 = \tf14 \xz(2)$ and $c_1 = \tf34 \xz(2)$, and it follows from Lemma \ref{F:lem:TailBB} that
    \<
		\label{F:eq:PsiXjExpr2}
        \Psi_j(z) = c_j\cX_j(z) + O\pfr{\cX_j(z)}{K} + O\pfr{X}{K} \qquad (z \in \xO).
    \>
    
If instead of \eqref{F:eq:FjCalXj} one has $F_j(z^k) \less \cX_j(z) / k$ for $z \in \xO$, then one concludes that 
    \[
        \Psi_j(z) \less \cX_j(z) + X/K \qquad (z \in \xO)
    \]
in place of \eqref{F:eq:PsiXjExpr2}. The above discussions are summarized in the following lemma. 

\begin{lemma}
    \label{F:lem:Fundamental}
    Let $f : \nn \to \signs$, and for $j\in\{0,1\}$ let $F_j(z)$ and $\Psi_j(z)$ be defined as in \eqref{F:eq:F0F1Defin} and \eqref{F:eq:Psi0Psi1Defin}, so that
        \[
            \Psi(z) = \sum_{k \equiv 0 \mod2} \fr{1}{k} F_0(z^k) + \sum_{k \equiv 1\mod2} \fr{1}{k} F_1(z^k) = \Psi_0(z) + \Psi_1(z) \qquad (|z|<1).
        \]
    Fixing $j\in\{0,1\}$, if for $z \in \xO$ and $1 \leq k \leq K$ one has $F_j(z^k) = \cX_j(z) / k$, then
        \<
            \label{F:eq:PsiXjExpr2lem}
            \Psi_j(z) = c_j\cX_j(z) + O\pfr{\cX_j(z)}{K} + O\pfr{X}{K} \qquad (z \in \xO),
        \>
	where $c_0 = \tf14 \xz(2)$ and $c_1 = \tf34 \xz(2)$. If instead of the relation $F_j(z^k)=\cX_j(z)/k$ one has that $F_j(z^k) \less \cX_j(z) / k$, then in place of equation \eqref{F:eq:PsiXjExpr2lem} one concludes that 
        \[
            \Psi_j(z) \less \cX_j(z) + X/K \qquad (z\in\xO).
        \]
\end{lemma}

\section{The minor arcs}
\label{sec:MinorArcs}

We now begin our specific analyses of $\Psi(z,\mu)$, so $\Psi(z)$ always indicates $\Psi(z,\mu)$ unless otherwise stated. As $\Psi_1(z)$ is easier to estimate than $\Psi_0(z)$, we analyze $\Psi_1(z)$ first.

\begin{lemma}
    \label{M:lem:Minor1}
        Fix $A > 0$. For $X > X_0(A)$ and $\xa \in \fm(X,Q)$ one has
            \[
                \Psi_1(\rho e(\xa)) \less X(\logX)^{-A}.
            \]
    \end{lemma}
    
\begin{proof}
    With $f = \mu$, the function $F_1(z)$ of \eqref{F:eq:F0F1Defin} is $\sum_{n=1}^\infty\mu(n) z^n$, whereby
        \[
            \Psi_1(\rho e(\xa)) = \sum_{k\equiv 1\mod2}  \frac{1}{k} F_1(\rho^k e(k\xa)) = \sum_{k\equiv 1\mod2} \frac{1}{k} \sum_{n=1}^\infty \mu(n) e^{-nk/X} e(n k\xa).
        \]
    Let $k \leq X^{1/2}$. Integrating by parts and recalling that $S_f(t,\theta)=\sum_{n \leq t} f(n)e(n\theta)$, we find that
        \<
            \label{M:eq:MajorF1Int}
            F_1(\rho^k e(k\xa)) = 
            \frac{k}{X} \int_0^\infty e^{-tk/X} S_\mu(t, k\xa) \,dt.
        \>
    When $t \geq X^{3/4}/k$, we bound $S_\mu(t,k\xa)$ using Davenport's inequality \cite{davenport1937some}:
        \<
            \label{eq:Davenport}
            S_\mu(t,\xq) \less t (\log t)^{-A} \qquad \text{for $t \geq 3$, uniformly in $\xq$.}
        \>
    Trivially bounding $S_\mu(t, k\xa) \less t$ when $t < X^{3/4}/k$, it follows that integral \eqref{M:eq:MajorF1Int} is
        \begin{align} 
            &\less \frac{k}{X} \lf[ \int_0^{X^{3/4}k^{-1}} e^{-tk/X} t \,dt + \int_{X^{3/4}k^{-1}}^\infty e^{-tk/X} t (\log t)^{-A} \,dt \rh] \notag \\
            \label{M:eq:MajorF1IntBB}
            &\less \frac{X}{k} \lf[ \int_0^{X^{-\nfr14}} e^{-u} u \,du + \log^{-A}(X^{\nfr34
            }k^{-1}) \int_{X^{-\nfr14}}^\infty e^{-u} u \,du \rh],
        \end{align}
    and we note that the assumption $k \leq X^{1/2}$ ensures that our use of inequality \eqref{eq:Davenport} remains valid as $X \to \infty$. We also note the change of $k/X$ to $X/k$ in expression \eqref{M:eq:MajorF1IntBB}. 
    As $\int_0^a e^{-u} u \,du \less a$ and $\int_a^\infty e^{-u} u \,du \less 1$ when $0<a<1$, we deduce that
        \[
            F_k(\xr^k e(k\xa)) \less \frac{X}{k}\Big[ X^{-\fr14} + \log^{-A}(X^{\fr34}k^{-1}) \Big] \less \frac{X/k}{(\logX)^{A}} \qquad (k \leq X^{\frac{1}{2}}).
        \]
    Applying Lemma \ref{F:lem:Fundamental} with $K = X^{1/2}$ we find that
        \[
            \Psi_1(\rho e(\xa)) \less X(\logX)^{-A} + X/X^{1/2},
        \]
    and the result follows. 
\end{proof}

\begin{remark}
    The reader may notice that the proof of Lemma \ref{M:lem:Minor1} did not utilize the assumption that $\xa \in \fm(X,Q)$. Indeed, because Davenport's inequality \eqref{eq:Davenport} is uniform in the variable $\xq$ (which became $k\xa$ above), we immediately deduce the following corollary.
\end{remark}

\begin{corollary}
    \label{M:lem:Psi1Uniform}
    For fixed $A>0$, real $\xa$, and all $X > X_0(A)$, one has that
        \[
            \Psi_1(\rho e(\xa)) \less X(\logX)^{-A}.
        \]
\end{corollary}

The analysis of $\Psi_0(\rho e(\xa))$ is much more involved than that of $\Psi_1(\rho e(\xa))$, as we require a suitable replacement for Davenport's inequality \eqref{eq:Davenport} for the sums $S_{\mu^2}(t,\xq)$. Our starting point is a theorem of Br\"udern and Perelli that is built on previous work of Br\"udern, Granville, Perelli, Vaughan, and Wooley \cite{brudern1998exponential}.

\begin{lemma}[{\cite{brudern1999exponential}, Thm.\ 4}]
    \label{lem:Brudern}
        If $1 \leq Q \leq X^{3/7}$, then
            \<
                \label{eq:Brudern}
                \sup_{\xq\in\fm(X,Q)} \lf| S_{\mu^2}(X, \xq)\rh| \less X^{1+\xe} Q^{-1}.
            \>
    \end{lemma}
 
To motivate how Lemma \ref{lem:Brudern} is used to establish our substitute for inequality \eqref{eq:Davenport} (and ultimately our bound on $\Psi_0$), we discuss a technical item to be addressed. Namely, to apply Lemma \ref{F:lem:Fundamental} we need that $F_0(\rho^k e(k\xa)) \less \cX_0(\rho^k e(k\xa)) / k$ for $k \leq K$, with $\cX_0$ not depending on $k$. However, the multiples $k\xa$ may not all be contained in $\fm(X,Q)$ as is necessary to apply the bound \eqref{eq:Brudern}. In fact, it may be that some $k\xa > 1$; however, as $S_f(t,\xq)=S_f(t,\{\xq\})$ for any $f$, it suffices to consider the fractional parts $\{k\xa\}$ of the $k\xa$.

Suppose for the moment that $K$, $\cX$, and $\cQ$ are quantities such that $\{k\xa\} \in \fm(\cX,\cQ)$ for all $\xa \in \fm(X,Q)$ and all $k \leq K$. In this case we have (for a general $f$) that
    \<
    \label{eq:minorInclBB}
    \sup_{\substack{\xa\in\fm(X,Q)\\ k \leq K}} |S_f(t, k\xa)| 
        = \sup_{\substack{\xa\in\fm(X,Q)\\ k \leq K}} |S_f(t,\{k\xa\})| 
        \leq \sup_{\xq \in \fm(\cX,\cQ)} |S_f(t,\xq)|.
    \>
Next, considering the definitions of the sets $\fM(\cX,\cQ)$ and $\fm(\cX,\cQ)$, one sees that if $t \geq \cX$, then $\fm(\cX,\cQ) \subset \fm(t,\cQ)$. From \eqref{eq:minorInclBB} then, for all $t \geq \cX$ we have that
    \<
        \label{eq:MinorInclBB1}
        \sup_{\substack{\xa\in\fm(X,Q)\\ k\leq K}} |S_f(t, k\xa)|
            \leq \sup_{\xq \in \fm(\cX,\cQ)} |S_f(t,\xq)|
            \leq \sup_{\xq\in\fm(t,\cQ)} |S_f(t,\xq)|.
    \>
Thus we need only ensure that we may apply Lemma \ref{lem:Brudern} to the final expression in \eqref{eq:MinorInclBB1} with $f=\mu^2$. These ideas are summarized in the following lemma. 

\begin{lemma}
    \label{lem:MinorInclBB}
    Suppose that $1\leq Q\leq X$, that $1\leq\cQ\leq\cX^{3/7}$, that $K>0$, and that: For all $\xa\in\fm(X,Q)$ and $k\leq K$ one has $\{k\xa\}\in\fm(\cX,\cQ)$. Then for $t\geq\cX$ one has
        \<
            \label{eq:MinorInclBrudern}
            \sup_{\substack{\xa\in\fm(X,Q)\\ k\leq K}} |S_{\mu^2}(t,k\xa)| 
                \leq \sup_{\xq\in\fm(t,\cQ)} |S_{\mu^2}(t,\xq)| 
                \less t^{1+\xe}\cQ^{-1}.
        \>
\end{lemma}

\begin{proof}
    Immediate by inequality \eqref{eq:MinorInclBB1} and Lemma \ref{lem:Brudern}.
\end{proof}

We now consider how to select $K$, $\cX$, and $\cQ$. The simple choices $\cX = X$ and $\cQ = Q$ yield the estimate $\Psi(\rho e(\xa)) \less X^{1+\xe}$, which is insufficient since $\Psi(\rho e(\xa))$ is comparable to $X$ when $\xa\in\fM(X,Q)$.  
Examining Lemma \ref{lem:Brudern} and the proof of Lemma \ref{M:lem:Minor1}, one may deduce that \emph{decreasing} $\cX$ has a far greater effect on the resulting estimates of $\Psi_0(\rho e(\xa))$ than does \emph{increasing} $\cQ$, so we prioritize decreasing $\cX$. The quantities $\cX$ and $\cQ$ are linked through the set $\fm(\cX,\cQ)$ however, so we must consider an ``exchange'' between them.  The mechanism of this exchange is shown in the following lemma.

\begin{lemma}
\label{lem:MinorIncl}
    Fix $\xk \in (0,1)$, let $1 \leq Q \leq X$, let $\xa \in \fm(X,Q)$, and let $1 \leq k \leq Q^\xk$. If $(a_k,q_k)=1$ and 
        \<
            \lf|q_k k\xa-a_k\rh| \leq \nfr{Q^{1-\xk}}{X},    
        \>
    then $q_k > Q^{1-\xk}$. In particular one has $\{k\xa\} \in \fm(X,Q^{1-\xk})$. Finally, if $0 < \nu < 1-\xk$ then 
        \<
            \label{eq:MinorInclusion}
            \fm(X,Q^{1-\xk}) \subset \fm(XQ^{-\nu},Q^{1-\xk-\nu}).
        \>
\end{lemma}

\begin{proof}
    We begin by recalling that $\fm(X,Q)$ consists of those $\xb \in [0,1)$ such that:\ if $0 \leq b < r$, $(b,r) = 1$, and $|r\xb-b| \leq Q/X$, then it holds that $r > Q$. 
    
    By Dirichlet's approximation theorem there exist $q_k \in \nn$ and $a_k \in \zz$ with $q_k \leq X/Q$, $(a_k,q_k)=1$, and $|q_kk\xa  - a_k| \leq Q/X$. By writing $a_k/(kq_k) = a_k'/(k'q_k)$ with $(a_k',k'q_k)=1$ and $k' \leq k$, we deduce that
        \[
            \lf| \xa-\frac{a_k'}{k'q_k} \rh| =  \lf| \xa-\frac{a_k}{k q_k} \rh| \leq \frac{Q}{kq_kX} \leq \frac{Q}{k'q_k X}.
        \]
    As $(a_k',k'q_k) = 1$ and $\xa \in \fm(X,Q)$, we thus have $k'q_k > Q$ and 
    {\binoppenalty=10000\relpenalty=10000 $q_k > Q/k' \geq Q^{1-\xk}$}. 
    Again using Dirichlet's approximation theorem, there exist $0 \leq b_k < r_k$ with $(b_k,r_k)=1$ and $|r_k\{k\xa\} - b_k| \leq Q/X$. Setting $a_k := b_k + r_k \lfloor k \xa \rfloor$, we have that $(a_k,r_k)=1$ and $|r_k k \xa - a_k| \leq Q/X$, whereby $r_k > Q^{1-\xk}$ and $\{k\xa\} \in \fm(X,Q^{1-\xk})$.

    Now let $0<\nu<1-\xk$ and $\xb \in \fm(X,Q^{1-\xk})$, and suppose that $(b,r)=1$ and
        \<
            \label{eq:DirichletIneq}
            |r\xb-b| \leq \nfr{Q^{1-\xk-\nu}}{(XQ^{-\nu})}.
        \> 
    Of course $|r\xb-b|\leq Q^{1-\xk}/X$ then, whereby $r > Q^{1-\xk} \geq Q^{1-\xk-\nu}$, yielding the result.  
\end{proof}

We now establish the desired substitute for Davenport's inequality \eqref{eq:Davenport}.

\begin{lemma}
    \label{M:lem:MinorSmu2FinalBB}
    Let $X$ be sufficiently large, and let $\xd,\xk,\nu \in (0,1)$ be such that $\nu < 1 - \xk$ and $\xd(1-\xk-\nu) \leq \tf37 (1-\xd\nu)$. Let $Q = X^\xd$. For all $x \geq XQ^{-\nu}$ one has
        \< 
            \sup_{\substack{\xa \in \fm(X,Q) \\ k \leq Q^{\xk}}} \lf| S_{\mu^2}(x,k\xa) \rh| \less x^{1+\xe} Q^{\xk+\nu-1}.
        \>
\end{lemma}

\begin{proof}
    By Lemma \ref{lem:MinorIncl}, we have $\{k\xa\} \in \fm(XQ^{-\nu}, Q^{1-\xk-\nu})$ for all $\xa \in \fm(X,Q)$ and $k \leq Q^\xk$. As $Q^{1-\xk-\nu} \leq (XQ^{-\nu})^{3/7}$ by our assumptions, we apply Lemma \ref{lem:MinorInclBB} with $\cX = XQ^{-\nu}$, $\cQ = Q^{1-\xk-\nu}$, and $K = Q^\xk$ to conclude that
        \<
            \label{eq:minorsupBB}
            \sup_{\substack{\xa \in \fm(X,Q) \\ k \leq Q^\xk}} \lf| S_{\mu^2}(t,k\xa) \rh| \less t^{1+\xe}(Q^{1-\xk-\nu})^{-1} \qquad (t \geq XQ^{-\nu}).
        \>
\end{proof}

At last we estimate $\Psi_0(\rho e(\xa))$, and thereby $\Psi(\rho e(\xa))$, for all $\xa \in \fm(X,Q)$. 

\begin{proposition}
	\label{M:prop:Minor}
	Let $A > 0$ be fixed and let $Q = X^{2/5}$. For all $X > X_0(A)$ and all $\xa \in \fm(X,Q)$ one has
		\[
			\Psi(\rho e(\xa)) \less X (\logX)^{-A}.
        \]
\end{proposition}

\begin{proof}
    Since $\Psi_1(\rho e(\xa)) \less X(\logX)^{-A}$ by Lemma \ref{M:lem:Minor1}, it is more than sufficient to show that $\Psi_0(\rho e(\xa)) \less X^{\nfr9{10}}$. Recalling from \eqref{F:eq:F0F1Defin} that $F_0(z) = \sum_{n=1}^\infty \mu^2(n) z^{n}$, we sum by parts to find that
		\[
            F_0(\rho^k e(k\xa)) = \sum_{n=1}^\infty \mu^2(n) \rho^{nk} e(nk\xa) = \frac{k}{X} \int_0^\infty e^{-tk/X} S_{\mu^2}(t, k\xa) \,dt.
		\]    
    
    For the sake of exposition we let $\xd,\xk,\nu \in (0,1)$ be unknowns such that $\nu < 1 - \xk$ and $\xd(1-\xk-\nu) \leq \tf37 (1-\xd\nu)$, let $Q = X^\xd$, and let $k \leq Q^{\xk}$. Bounding $S_{\mu^2}(t,k\xa)\less t$ when $t \leq XQ^{-\nu}$ and using \eqref{eq:minorsupBB} when $t > XQ^{-\nu}$, we have
        \begin{align*}
            \notag
            F_0(\rho^k e(k\xa)) 
                &\less \frac{k}{X} \brk{ \int_0^{XQ^{-\nu}} e^{-tk/X} t \,dt + Q^{\xk+\nu-1} \int_{XQ^{-\nu}}^\infty e^{-tk/X} t^{1+\xe} \,dt } \\
            \notag
            &= \frac{X}{k} \brk{ \int_0^{kQ^{-\nu}} e^{-u} u \,du  + Q^{\xk+\nu-1} \pfr{X}{k}^\xe \int_{kQ^{-\nu}}^\infty e^{-u}u \,du } \\
            &\less \frac{X}{k} \Big[ 1 - \lf( 1+kQ^{-\nu} \rh) \expp{-kQ^{-\nu}} + Q^{\xk+\nu-1} X^\xe \Big].         
        \end{align*}
    As $k \leq Q^\xk$ and the function $1-(1+t)e^{-t}$ is increasing in $t$, this is
        \<
            \label{M:eq:MinorF0BB3}
            \less \frac{X}{k} \Big[ 1 - (1+Q^{\xk-\nu}) \exp(-Q^{\xk-\nu}) + Q^{\xk+\nu-1} X^\xe \Big].     
        \>
    If, in addition to our previous assumptions, we suppose that $\xk-\nu<0$, then the inequality $1-(1+t)e^{-t} \less t^2$ (for small $t$), and the fact that $Q$ may be assumed to be large (since $Q=X^\xd$ and $X$ is large), together imply that \eqref{M:eq:MinorF0BB3} is
        \<
            \label{M:eq:MinorF0BB4}
            \less \fr{X}{k} \Big[ Q^{2(\xk-\xn)} + Q^{\xk+\xn-1} X^\xe \Big] = \fr{X}{k} \Big[ X^{-2\xd(\xn-\xk)} + X^{-\xd(1-\xk-\xn)+\xe} \Big].
        \>
    Taking $\xe < \tf15\xd(1-\xk-\nu)$, we deduce that
        \<
            \label{M:eq:MinorF0cX}
            F_0(\rho^k e(k\xa)) \less \fr{1}{k} \lf[ X^{1-2\xd(\xn-\xk)} + X^{1-\fr45\xd(1-\xk-\xn)} \rh],
        \>

    and applying Lemma \ref{F:lem:Fundamental} using \eqref{M:eq:MinorF0cX} and $K = Q^\xk = X^{\xd\xk}$, it follows that
		\[
			\Psi_0(\rho e(\xa)) \less X^{1-2\xd(\xn-\xk)} + X^{1-\fr45\xd(1-\xk-\xn)} + X^{1-\xd\xk}.
        \]
    Setting $\xd = \tf25$, $\xk = \tf14$, and $\nu = \tf25$, we see that these satisfy the conditions
        \[
            \nu < 1-\xk, \quad \xk < \nu, \quad\text{and}\quad \xd(1-\xk-\xn) \leq \tf37(1-\xd\xn)   
        \] 
    and yield the bound
        \[
            \Psi_0(\rho e(\xa)) \less X^{\fr{17}{25}} + X^{\fr{111}{125}} +  X^{\fr{9}{10}} \less X^{\fr{9}{10}},
        \]
    and the proposition follows.
\end{proof}

\section{The major arcs}

We now analyze $\Psi(\rho e(\xa))$ for $\xa \in \fM(X,Q)$, which includes when $\xa$ is in $\fP$ or $\fP_*$. Although the choice $Q = X^{2/5}$ mentioned in the proof of Proposition \ref{M:prop:Minor} is ultimately used, we again treat $Q$ as a free parameter until this is choice is necessary.
In truth, we need not further analyze $\Psi_1(\rho e(\xa))$ since Corollary \ref{M:lem:Psi1Uniform} provides a satisfactory estimate already. While our analysis of $\Psi_0(\rho e(\xa))$ on the minor arcs requires estimates of the exponential sums $S_{\mu^2}(t,\xq)$, our analysis on the major arcs requires results on sums of $\mu^2(n)$ as $n$ runs over arithmetic progressions. 

We recall that $\mu(n) \neq 0$ if and only if $p^2 \nmid n$ for all prime $p$, and that such $n$ are said to be squarefree or \emph{quadratfrei}. For positive $q$ and $r$, not necessarily coprime, we define
    \[
        \cQ(t;q,r) = \sum_{\substack{n\leq t\\ n \equiv r \mod{q}}} \mu^2(n), 
    \]
which is simply the number of squarefree $n \leq t$ that are congruent to $r$ modulo $q$.
The following function determines the asymptotic behavior of $\cQ(t;q,r)$, so it is convenient to define it now.

\begin{definition}
    If $(q,r)$ is squarefree set
        \<
            \label{eq:g}
            g(q,r) := \sum_{\substack{n=1 \\ (n^2,q)\;|\;r}}^\infty \frac{\mu(n)(n^2,q)}{n^2 q},
        \>
    and set $g(q,r):=0$ otherwise. 
\end{definition}

We note that \eqref{eq:g} immediately implies that $g(q,r) = g(q,(q,r))$. Although the following result is well known, we include a proof here for completeness.

\begin{lemma}
    \label{M:lem:SqFreeAsymp}
    For all positive integers $q$ and $r$ and all sufficiently large $t$, one has
	    \[
		    \cQ(t;q,r) = g(q,r) t + O(t^{1/2})
        \]
    with implicit constant independent of $q$ and $r$.
\end{lemma}

\begin{proof}
    Writing $n=ab^2$ with $a$ squarefree, and using the well known identity 
        \[
            \mu^2(n) = \sum_{d^2\mid n} \mu(d),
        \] 
    we rearrange terms in the sum $\cQ(t;q,r)$ to find that
        \[ 
        \cQ(t;q,r) 
            = \sum_{b \leq \rt{t}} \mu(b) \bigg( \sum_{\substack{a \leq t/b^2 \\ ab^2 \equiv r \mod{q}}} 1 \bigg)
            = \sum_{d\;|\;r} \sum_{\substack{b\leq\rt{t} \\ (b^2,q)=d}} \mu(b) \bigg(\sum_{\substack{a \leq t/b^2 \\ ab^2 \equiv r\mod{q}}} 1 \bigg).
        \]
    For any $b$ let $d=(b^2,q)$, let $b^2=\xb d$, let $q=\xq d$, and let $r=\xr d$. Then $ab^2\equiv r\mod{q}$ if and only if $a\xb \equiv\xr\mod{\xq}$, and defining $\xb'$ so that $\xb'\xb \equiv 1 \mod{\xq}$ (valid since $(\xb,\xq)=1$), we have that $ab^2 \equiv r \mod{q}$ if and only if $a \equiv \xb'\xr\mod{\xq}$. It follows that
        \[ \cQ(t;q,r)  
            = \sum_{d\;|\;r} \sum_{\substack{b\leq\rt{t} \\ (b^2,q)=d}} \mu(b) \bigg(\sum_{\substack{a \leq t/b^2 \\ a \equiv \xb'\xr\mod{\xq}}} 1\bigg) 
            = \sum_{d\;|\;r} \sum_{\substack{b\leq\rt{t} \\ (b^2,q)=d}} \mu(b) \lf(\frac{t}{b^2\xq} + \O{1} \rh).
        \]
    Writing $1/(b^2\xq) = d/(b^2q)$ we deduce that 
        \<  
            \label{M:eq:cQAsymp}
            \cQ(t;q,r) = t \sum_{\substack{b\leq\rt{t} \\ (b^2,q)\;|\;r}} \frac{\mu(b)(b^2,q)}{b^2q} + O(t^{1/2}),
        \>
    noting the implicit constant to be independent of $q$ and $r$. Extending the sum over $b$ in equation \eqref{M:eq:cQAsymp} to sum over all $b \geq 1$ at a cost of $O(t^{-1/2})$, we conclude that
        \[
            \cQ(t;q,r) = t \bigg(\sum_{\substack{n=1 \\ (n^2,q)\;|\;r}}^\infty \frac{\mu(n)(n^2,q)}{n^2 q}\bigg) + \O{t^{1/2}}
        \]
    as claimed, and the implicit constant is again independent of $q$ and $r$. 
\end{proof}

We now record a lemma on an exponential sum involving the constants $g(q,r)$.

\begin{definition}
    Let $G(q)$ be the multiplicative function defined on prime powers by         
        \[
            G(p^i)= 
            \begin{cases}
                - (p^2 - 1)^{-1} & i=1,2, \\
                0 & \text{otherwise.}
            \end{cases}
        \] 
\end{definition}

As shown in \cite{vaughan2005variance}*{Lem.\ 2.5}, for all $q$ one has
    \<
        \label{eq:gsum}
        \sum_{r = 1}^q e(r/q) g(q,r) = \xz(2)^{-1}G(q).
    \>
We now derive the primary asymptotic formula for our major arcs. 

\begin{proposition}
    \label{M:prop:Major}
    Fix $A>0$, let $Q=X^{2/5}$, and suppose that $\xa = a/q+\xb$ with $0\leq a \leq q \leq Q$, $(a,q)=1$, and $|\xb| \leq Q/(qX)$. There exists a real constant $\Vmaj(q)$ such that, as $X\to\infty$, one has
        \<
            \label{M:eq:Major}
            \Psi(\rho e(\xa)) = \Vmaj(q)\pfr{X}{1-2\pi iX\xb}  + O(X (\logX)^{-A}),
        \>
    with implicit constant independent of $q$ and $a$, and with $V(q)$ given in the following manner. Namely, letting $\tilde{V}(q)$ be the multiplicative function defined on prime powers $p^k>1$ via
        \<
            \label{eq:Vtilde}
            \tilde{V}(p^k) = \begin{cases}
                0 & k=1, \\
                -p^2 & k \geq 2,
            \end{cases}
        \>
    one has
        \<
            \label{eq:VVtilde}
            V(q) = \begin{cases}
                (2q)^{-2} \tilde{V}(q) & 2 \nmid q, \\
                q^{-2} \tilde{V}(q/2) & 2 \mid q.
            \end{cases}
        \>
    In particular, one has $\Vmaj(1) = \Vmaj(2) = \tfrac{1}{4}$, $\Vmaj(4)=0$, and $|\Vmaj(q)| \leq \frac{1}{16}$ for all other $q$.
\end{proposition}

\begin{proof}[Proof of Proposition \ref{M:prop:Major}.]  
    Let $\tau := X^{-1}(1-2\pi iX\xb)$ so that $\rho e(\xa) = e(a/q)e^{-\tau}$. As $\Psi_1(\rho e(\xa)) \less X(\logX)^{-A}$ for $X > X_0(A)$, it suffices to show that
    	\<
            \label{M:eq:MajorPsi0Asymp}
    		\Psi_0(\rho e(\xa)) = \Vmaj(q) \tau^{-1} + O(X^{\fr{9}{10}}).
    	\>
    For all $k \in \nn$ set
        \[
            q_k := q / (q,k) \qqand a_k := a k / (q,k).
        \]
    Recalling definition \eqref{F:eq:F0F1Defin} we have $F_0(z) = \sum_{n=1}^\infty \mu^2(n) z^n$, and grouping summands into residue classes modulo $q_k$ we have
		\<
			\label{M:eq:MajorF0Term}
			F_0( \rho^k e(k\xa) )  = \sum_{n=1}^\infty \mu^2(n) e\lf(\frac{nka}{q}\rh) e^{-nk \tau} 
			= \sum_{r =1}^{q_k}  e\pfr{a_k r}{q_k} \sum_{ n \equiv r \mod{q_k}} \mu^2(n) e^{-nk \tau}.
		\>
	Summing by parts and applying Lemma \ref{M:lem:SqFreeAsymp}, we find that
		\<
            \label{M:eq:MajorF0Int}
			\sum_{n \equiv r \mod{q_k}} \mu^2(n)e^{-nk\tau} = k\tau \int_0^\infty e^{-tk\tau}\cQ(t;q_k,r) \,dt =   k\tau\int_0^\infty e^{-tk\tau}(tg(q_k,r)+O(t^{\fr12})) \,dt.
		\>
    As $|\xt| \less X^{-1}(1 + X|\xb|) \less Q/(qX)$, the error term here is
		\<
            \label{M:eq:MajorIntErr}
			k\tau \int_0^\infty e^{tk\tau}O(t^{1/2}) \less \frac{kQ}{qX} \int_0^\infty e^{-tk/X} t^{1/2} \,dt \less \frac{kQ}{qX} \pfr{X}{k}^{3/2} = \frac{Q}{q}\pfr{X}{k}^{1/2},
		\>
	and we remark that it is in equation \eqref{M:eq:MajorF0Int} and inequality \eqref{M:eq:MajorIntErr} that the independence (from $q_k$ and $r$) of the constant in Lemma \ref{M:lem:SqFreeAsymp} is used.

    To consider the remaining integral in equation \eqref{M:eq:MajorF0Int}, we first rewrite
		\<
            \label{M:eq:MajorMainInt}
			g(q_k, r) \int_0^\infty e^{-t k \tau}  t k \tau \,dt 
			= (k \tau)^{-1} g(q_k, r) \int_0^\infty e^{-t k \tau }  t k \tau \,d(t k \tau).
		\>
    Observing that $|\arg k\tau| = |\arg(1-2\pi iX\xb)| < \fr{\pi}{2}$ for all $|\xb|\leq Q/qX$, we recognize the integral in \eqref{M:eq:MajorMainInt} to be equal to $\xG(2)$ (see, e.g., \cite{montgomery2007multiplicative}*{App.\ C}), and conclude that
        \<
            \label{M:eq:MajorGammaInt}
            g(q_k,r) \int_0^\infty e^{-tk\tau} tk\tau \,dt = \frac{g(q_k,r)}{k \tau}.
        \>

    Returning to equation \eqref{M:eq:MajorF0Term} and using \eqref{M:eq:MajorIntErr} and \eqref{M:eq:MajorGammaInt}, we find that
    	\<
            \label{M:eq:MajorF0Sum2}
    		\sum_{r =1}^{q_k} e\pfr{a_k r}{q_k} \sum_{n \equiv r \mod{q_k}} \mu^2(n)  e^{-nk \tau}
    	    = \frac{1}{k \tau} \sum_{r =1}^{q_k} e\pfr{a_k r}{q_k}  g(q_k, r) + \O{\frac{X^{1/2}Q}{k^{1/2}}}.
    	\>
	As $(a_k, q_k) = 1$, we define $a_k'$ so that $a_k a_k' \equiv 1 \mod{q_k}$ and recall from \eqref{eq:g} that $g(q,r)=g(q,(q,r))$ to deduce that
		\<
            \label{eq:gGcdSwap}
			g(q_k, r) = g(q_k, (q_k, r)) = g(q_k, (q_k, r a_k')) = g(q_k, r a_k'). 
      	\>
	By equation \eqref{eq:gGcdSwap} one has
		\<
            \label{eq:gExpSum}
			\sum_{r =1}^{q_k} e(\nfr{a_k r}{q_k})  g(q_k, r) 
			= \sum_{r =1}^{q_k} e(r/q_k)  g(q_k, r a_k')
			= \sum_{r =1}^{q_k} e(r/q_k)  g(q_k, r),
		\>
	which by equation \eqref{eq:gsum} is equal to $\xz(2)^{-1}G(q_k)$. 
    
    Combining equations \eqref{M:eq:MajorF0Int}, \eqref{M:eq:MajorF0Sum2}, and \eqref{eq:gExpSum} we conclude that
		\<
          \label{M:eq:MajorF0Sum3}
		    F_0(\rho^k e(k\xa)) =  \frac{G(q_k)}{\xt k \xz(2)} + O\pfr{X^{1/2}Q}{k^{1/2}},
      \>
	but we note that Lemma \ref{F:lem:Fundamental} cannot be applied here, as the quantity $\cX_0(z)$ indicated by \eqref{M:eq:MajorF0Sum3} is dependent on $k$. Thus considering the definition of $\Psi_0(z)$ directly, it follows from \eqref{M:eq:MajorF0Sum3} that
		\<
            \label{M:eq:MajorPsi0}
            \Psi_0(\rho e(\xa)) = \sum_{k \equiv 0\mod2} \fr{1}{k} F_0 ( \rho^k e(k\xa)) 
            = (\xt\xz(2))^{-1} \sum_{k \equiv 0\mod2} \frac{G(q_k)}{k^2} + 
            O(X^{1/2}Q), 
		\>
    and we are motivated to define
        \<
            \label{eq:VDefin}
            \Vmaj(q) = \xz(2)^{-1} \sum_{k \equiv 0 \mod2} \frac{G(q_k)}{k^2} 				= \tfrac{1}{4} \xz(2)^{-1} \sum_{k=1}^\infty G\left(\frac{q}{(q,2k)}\right) k^{-2}.
        \>

    So that the $O(X^{1/2}Q)$ term in equation \eqref{M:eq:MajorPsi0} is $o(X)$, it suffices to take $Q=X^{2/5}$, and equation \eqref{M:eq:MajorPsi0Asymp} (and thereby \eqref{M:eq:Major}) is established. We note that the implicit constants throughout our proof of equation \eqref{M:eq:MajorPsi0Asymp} are independent of $q$ and $a$, and therefore the implicit constant in equation \eqref{M:eq:Major} is independent of $q$ and $a$ as well.
    
    Finally we verify the claimed properties of $V(q)$. If $2 \nmid q$ then $(q,2k)=(q,k)$, and so grouping summands by $(q,k)$ we have
        \[
            V(q) = \fr{3}{2\pi^2} \sum_{d | q} G\left(\frac{q}{d}\right) \sum_{\substack{k=1 \\ (q,k)=d}}^\infty \frac{1}{k^2} = \fr{3}{2\pi^2} \sum_{d | q} \frac{G(q/d)}{d^2} \sum_{\substack{k=1 \\ (q/d,k)=1}}^\infty \frac{1}{k^2}.
        \]
    This latter quantity is equal to
        \[
            \tf14 \sum_{d | q} d^{-2}G(q/d) \prod_{p | qd^{-1}}(1-p^{-2}) = \tf14 q^{-2} \sum_{d | q} d^2G(d) \prod_{p | d}(1-p^{-2}) = \tf14 \tilde{V}(q).
        \]
	If $2 \mid q$, say $q = 2r$, then $q/(q,2k) = r/(r,k)$ and we similarly find that
        \[
            V(q) = \frac{3}{2\pi^2} \sum_{d | r} G\left(\frac{r}{d}\right) \sum_{\substack{k=1 \\ (r,k)=d}}^\infty \frac{1}{k^2} = (2r)^{-2} \sum_{d | r} d^2G(d) \prod_{p | d}(1-p^{-2}).
        \]
    
    Thus, setting
        \[
            \tilde{V}(n) := \sum_{d | n} d^2G(d) \prod_{p | d}(1-p^{-2}),
        \]
    it is apparent that $\tilde{V}$ is multiplicative (since $G$ is), and we deduce that 
        \[
            V(q) = \begin{cases}
                (2q)^{-2} \tilde{V}(q) & 2 \nmid q, \\
                q^{-2} \tilde{V}(q/2) & 2 \mid q.
            \end{cases}
        \]
    For primes $p$ we see that
        \[
            \tilde{V}(p) = 1 + p^2G(p)(1-p^{-2}) = 1 - 1 = 0,
        \]
    and for prime powers $p^k$ with $k \geq 2$ we see that
        \[
            \tilde{V}(p^k) = 1 + p^{2}G(p)(1-p^{-2}) + p^4G(p^2)(1-p^{-2})= 1 - 1 - p^2 = -p^2,
        \]
    thereby validating formulae \eqref{eq:Vtilde} and \eqref{eq:VVtilde}. The facts that $V(1)=V(2)=\tf14$, that $V(4)=0$, and that $|V(q)| \leq \tf{1}{16}$ for all other $q$ are verified at once. 
\end{proof}

\section{The principal arcs I}
\label{sec:PrincArcsI}

We recall that the principal arcs are
    \begin{align*}
        \fP &= \{ \xa \in [0,1) : \|\xa\| \leq X^{-1}(\logX)^{-\nfr14} \}, \\
        \fP_* &= \{ \xa \in [0,1) : |\xa-\tf12| \leq X_*^{-1}(\logX_*)^{-\nfr14} \}.
    \end{align*}
For brevity, let
    \<
        \xh := X^{-1}(\logX)^{-\nfr14} \qqand \xh_* := X_*^{-1}(\logX_*)^{-\nfr14}.
    \>
Because $\Psi(r e(\xa)) = \Psi(r e(\xa-1))$ for all $\xa$, it is convenient to identify the sets
    \<
        \label{eq:prcId}
        \{ e(\xa) : \xa \in \fP \} \qqand \{ e(\xb) : \text{$\xb\in\rr$ and $|\xb|\leq\xh$} \}
    \>
and consider $\Psi(\rho e(\xa))$ with $\xa \in \fP$ as $\Psi(\xr e(\xb))$ with $|\xb| \leq \xh$. As $e(1/2)=-1$, we may similarly identify the sets
    \<
        \label{eq:prc*Id}
        \{ e(\xa) : \xa \in \fP_* \} \qqand \{ -e(\xb) : \text{$\xb\in\rr$ and $|\xb|\leq\xh_*$} \}   
    \>
and consider $\Psi(\xr_* e(\xa))$ with $\xa\in\fP_*$ as $\Psi(-\xr_* e(\xb))$ with $|\xb|\leq\xh_*$.

Despite the above definitions' in terms of $X_*$, as discussed in section \ref{sec:prelim} we do not presently distinguish $X$ and $X_*$, because this and the following section do not involve solutions to the saddle-point equations \eqref{eq:sdp} and \eqref{eq:sdp*}. Considering this with \eqref{eq:prcId} and \eqref{eq:prc*Id} then, we see that our results on $\Psi(\pm\xr e(\xb))$ with $|\xb| \leq \xh$ sufficiently cover the analyses of $\Psi(\xr e(\xa))$ with $\xa$ in the principal arcs.

We now introduce some convenient notation for the analyses of $\Psi(\xr e(\xa))$ when $\xa$ is in the principal arcs. Because the discussions below are applicable to general $f:\nn\to\signs$, we temporarily return to discussion of a general $\Psi(z)=\Psi(z,f)$.

For complex $s$ we write $s=\xs+it$, and $\xs$ and $t$ always represent the real and imaginary parts of some $s\in\cc$, respectively. Integrals $\frac{1}{2\pi i} \int$ are abbreviated using dashed integrals $\dint$. Let $\Dir{f}{s}$ denote the Dirichlet series $\sum_{n=1}^\infty f(n) n^{-s}$ and let $\xs_0 = \xs_0(f)$ and $\xs_1 = \xs_1(f)$ denote the abcissae of absolute convergence for $\Dir{f^2}{s}$ and $\Dir{f}{s}$, respectively. For simplicity all $\xs_0$ and $\xs_1$ below are assumed to be positive. 
We recall from the definitions \eqref{F:eq:F0F1Defin} and \eqref{F:eq:Psi0Psi1Defin} that
    \[
        \Psi(z) = \Psi_0(z) + \Psi_1(z) = \sum_{\substack{k \equiv 0\mod2}} \sum_{n=1}^\infty \frac{f^2(n)}{k} z^{nk} + \sum_{k\equiv 1\mod2} \sum_{n=1}^\infty \frac{f(n)}{k} z^{nk}.
    \]

For real $\xb$ and $z=\xr e(\xb)$ we again set $\tau := \fr{1}{X}(1-2\pi iX\xb)$ so that $z= e^{-\xt}$. Using the well known Cahen-Mellin formula
    \[
        e^{-w} = \ldint{\xs} \xG(s)z^{-s} \,ds \qquad (\Re w>0,\,\,\xs > 0),
    \]
we express $\Psi_0(z)$ and $\Psi_1(z)$ as sums of integrals, and interchange the summations and integrations using the series' absolute convergence, to derive the formulae
	\begin{align}
        \label{M:eq:Psi0Mellin}
	    \Psi_0(z) = \Psi_0(e^{-\xt}) &= \ldint{1+\xs_0} 2^{-s-1}\xz(s+1) \Dir{f^2}{s} \xG(s)\tau^{-s} \,ds, \\
        \label{M:eq:Psi1Mellin}
	    \Psi_1(z) = \Psi_1(e^{-\xt}) &= \ldint{1+\xs_1} (1-2^{-s-1}) \xz(s+1) \Dir{f}{s} \xG(s)\tau^{-s} \,ds.
    \end{align}
We note that if $f$ is nonnegative the distinction between $\Psi_0$ and $\Psi_1$ is unnecessary, and we derive a formula like formula (6.2.6) of \cite{andrews1976partitions} as expected. In general, recalling that $\PsiD{j}(z):= \mop{z}^j\Psi(z)$, for $j \geq 0$ one has
	\<
        \label{eq:MopjPsi}
		\PsiD{j}(z) = \sum_{k=1}^\infty \sum_{n=1}^\infty k^{j-1} n^{j} f^k(n) z^{nk},
	\>
yielding (after a change of variable of $s$ to $s+j$) the general formulae
	\begin{align}
		\label{eq:PsiMellin0}
        (\Psi_0)_{(j)}(z) &= \ldint{1+\xs_0} 2^{-s-1} \xz(s+1) \fD[f^2;s]\xG(s+j)\tau^{-s-j} \,ds, \\
        \label{eq:PsiMellin1}
        (\Psi_1)_{(j)}(z) &= \ldint{1+\xs_1} (1-2^{-s-1}) \xz(s+1) \Dir{f}{s} \xG(s+j)\tau^{-s-j} \,ds.
    \end{align}

When $z = -\rho e(\xb) = -e^{-\tau}$ our derivations require only small changes. As $(-z)^{nk} = z^{nk}$ when $k$ is even, one sees from equation \eqref{eq:MopjPsi} that $\Psi_0(z)$ and the derivatives $(\Psi_0)_{(j)}(z)$ are even functions, whereby \eqref{eq:PsiMellin0} may be reused for $(\Psi_0)_{(j)}(-z)$. When $k$ is odd the equality $(-1)^{nk} = (-1)^n$ shows that 
    \<
        \label{eq:Psi1Neg1}
        (\Psi_1)_{(j)}(-z) = \sum_{k\equiv 1\mod2} \sum_{n=1}^\infty (-1)^n n^{-j} k^{-j-1} f(n) z^{nk}, 
    \>
and one derives a formula like \eqref{eq:PsiMellin1} by replacing $\Dir{f(n)}{s}$ with $\Dir{(-1)^n f(n)}{s}$.

\section{The principal arcs II}
\label{sec:PrincArcsII}

We now return to the specific study of $\Psi(z,\mu)$, again abbreviating it to $\Psi(z)$. In addition we maintain the definitions $\eta = X^{-1}(\logX)^{-1/4}$ and $\xt = X^{-1}(1-2\pi iX\xb)$ so that $\pm\xr e(\xb) = \pm e^{-\xt}$. The goal of this section is the proof of the following proposition.

\begin{proposition}
    \label{M:prop:Princ}
    Fix $A > 0$. For all $j \geq 0$, all $X > X_0(A)$, and all $|\xb|\leq \eta$ one has
        \<
            \label{M:eq:Momop}
		    \PsiD{j}(\pm \rho e(\xb)) = \frac{j!}{4} \pfr{X}{1-2\pi iX\xb}^{j+1} + O_j(X^{j+1}(\logX)^{-A})
	    \>
    and
        \<
            \label{M:eq:Deriv}
            (\pm1)^j\PsiD{j}(\pm \rho e(\xb)) = \frac{j!}{4} \pfr{X}{1-2\pi iX\xb}^{j+1} + O_j(X^{j+1}(\logX)^{-A}).
        \>
\end{proposition}

We record two useful, classical properties of $\xz(s)$ and $\xG(s)$. First, there exist universal positive constants $c$ and $t_0$ such that $\xz(s) \neq 0$ when
    \<
        \label{eq:ZetaZeroFree}
        |t|>t_0 \qqand \xs \geq 1 - \frac{c}{\log|t|},
    \>
and on this region one has \cite{montgomery2007multiplicative}*{Thm.\ 6.6 and 6.7}
    \<
        \label{eq:ZetaBB}
        |\xz(\xs+it)|^{-1} \less \log{t} \less t^\xe.
    \>
Second, one has \cite{paris2001asymptotics}*{ineq.\ (2.1.19)}
\<
    \label{eq:GammaBB}
    |\xG(s)| \less |s|^{\xs-\fr12} \expp{-\tf12 \pi|t| + \tf16 |s|^{-1}} \qquad (\xs>0). 
\>

\begin{lemma}
	\label{M:lem:PrincPsi1}
	Fix $A>0$. For all $j \geq 0$, all $X > X_0(A)$, and all $|\xb|\leq \xh$, one has
		\[
			(\Psi_1)_{(j)}(\rho e(\xb)) \less_j X^{j+1} (\logX)^{-A}.
		\]
\end{lemma}

\begin{proof}
	As $\xr e(\xb) = e^{-\xt}$ and $\Dir{\mu}{s} = 1/\xz(s)$, we have from \eqref{eq:PsiMellin1} that
        \<
            \label{M:eq:Psi1Int}
            (\Psi_1)_{(j)}(\rho e(\xb)) = \dint_{2-i\infty}^{2+i\infty} (1 - 2^{-s-1})\frac{\xz(s+1)}{\xz(s)}\xG(s+j)\tau^{-s-j} \,ds.
        \>
    
    Letting $T$ be a sufficiently large parameter of our choosing, we first truncate integral \eqref{M:eq:Psi1Int} at height $T$ with error term
        \<
            \label{M:eq:Psi1TruncError}
            \less \int_T^\infty \lf| \frac{\xG(2+j+it)\tau^{-j-2-it}}{\xz(2+it)} \rh| \,dt.
        \>
    Having assumed that $|\xb| \leq \eta = X^{-1}(\logX)^{-1/4}$, for sufficiently large $X$ one has
		\[
			|\arg\tau| = |\arg(1-2\pi iX\xb)| 
            \leq |\arctan(2\pi(\logX)^{-\nfr14})| \leq \fr{\pi}{4},
		\]
	where $\arctan(x)$ has values in $(-\tf{\pi}{2},\tf{\pi}{2})$, and since $|\xt|^{-1} \leq X$ this implies that
		\<
            \label{eq:tauAbsBB}
			|\tau^{-2-j-it}| = |\tau|^{-j-2} e^{t \arg \tau} \leq X^{j+2} e^{\pi t/4} \qquad (t>0). 
		\>
	Using inequalities \eqref{eq:ZetaBB}, \eqref{eq:GammaBB}, and \eqref{eq:tauAbsBB}, it follows that integral \eqref{M:eq:Psi1TruncError} is
		\[
			\less_j X^{j+2} \int_T^\infty t^{j+\nfr32+\xe} e^{-\pi t/4} \,dt \less_j X^{j+2} e^{-\pi T/5},
        \]
    whereby
        \<
            \label{M:eq:Psi1Trunc}
            (\Psi_1)_{(j)}(\rho e(\xb)) = \dint_{2-iT}^{2+iT} (1 - 2^{-s-1}) \frac{\xz(s+1)}{\xz(s)}\xG(s+j)\tau^{-s-j} \,ds + O_j( X^{j+2} e^{-\pi T/5}) 
        \>
	Let $\xs_* := 1 - c / \log T$, where $c$ is the constant of the zero free region \eqref{eq:ZetaZeroFree}. Using Cauchy's theorem, we see that the integral in \eqref{M:eq:Psi1Trunc} is
        \[
            = \lf[ \dint_{2-iT}^{\xs_*-iT} +\,\, \dint_{\xs_* -iT}^{\xs_*+iT} +\,\, \dint_{\xs_*+iT}^{2+iT} \rh] (1-2^{-s-1})\frac{\xz(s+1)}{\xz(s)}\xG(s+j)\tau^{-s-j} \,ds.
        \]
    Again using \eqref{eq:ZetaBB}, \eqref{eq:GammaBB}, and \eqref{eq:tauAbsBB}, we see that the integral from $\xs_*+iT$ to $2+iT$ here is
        \<
            \label{M:eq:Psi1Int24BB}
            \less \int_{\xs_*}^2 \lf| \fr{\xG(j+\xs+iT)\tau^{-j-\xs-iT}}{\xz(\xs+iT)} \rh| \,d\xs \less_j X^{j+2} T^{j+\nfr32+\xe} e^{-\pi T/4} \less_j X^{j+2} e^{-\pi T/5},
        \>
    and similarly for the integral from $2-iT$ to $\xs_*-iT$. Using similar arguments for the integral from $\xs_*-iT$ to $\xs_*+iT$, we find that
        \<
            \label{M:eq:Psi1Int3BB}
            \dint_{\xs_* -iT}^{\xs_* +iT} (1 - 2^{-s-1}) \frac{\xz(s+1)}{\xz(s)}\xG(s+j)\tau^{-s-j} \,ds \less_j X^{j+1-c/\log T} \log T.
        \>
    
    Combining equation \eqref{M:eq:Psi1Trunc} and inequalities \eqref{M:eq:Psi1Int24BB} and \eqref{M:eq:Psi1Int3BB}, we deduce that
		\<
            \label{eq:Psi1CrudeBB}
			(\Psi_1)_{(j)}(\rho e(\xb)) \less_j X^{j+1} \lf( X e^{-\pi T/5} + X^{-c/\log{T}} \log{T}\rh).
        \>
    Now letting $T=\exp(\rt{\logX})$, we have that
        \[
            X \expp{-\fr{\pi}{5} T} = \expp{\logX - \fr{\pi}{5} e^{\rt{\logX}} } \less \expp{ -\fr{\pi}{6} e^{\rt{\logX}} }
        \]
    and
        \[
            X^{-c/\log{T}} \log{T} 
            = \lf|\frac{X}{1-2\pi iX\xb} \rh|^{-c/\rt{\logX}} \rt{\logX} 
            \less e^{-c \rt{\logX}} \rt{\logX},
        \]
    and the result follows from inequality \eqref{eq:Psi1CrudeBB} since
        \[
            \expp{-\fr{\pi}{6} e^{\rt{\logX}}} + e^{-c\rt{\logX}} \rt{\logX} \less (\logX)^{-A}.
        \]
\end{proof}

The analysis of $(\Psi_0)_{(j)}$ is largely similar to that of $(\Psi_1)_{(j)}$. As $\Dir{\mu^2}{s} = \xz(s) / \xz(2s)$, equation \eqref{eq:PsiMellin0} becomes
    \<
        \label{M:eq:Psi0AsympInt}
        (\Psi_0)_{(j)}(\rho e(\xb)) = \ldint{2} 2^{-s-1} \xz(s+1) \frac{\xz(s)}{\xz(2s)} \xG(s+j)\tau^{-s-j} \,ds,
    \>
and we see that the integrand here has a simple pole at $s = 1$ with residue
    \<
        \label{eq:mures}
        2^{-2} \frac{\xz(2)}{\xz(2)} \xG(j+1)\tau^{-j-1} = \frac{j!}{4}\pfr{X}{1-2\pi iX\xb}^{j+1}.
    \>
Under the same assumptions as in Lemma \ref{M:lem:PrincPsi1}, we account for this residue and shift the line of integration in \eqref{M:eq:Psi0AsympInt} left to $\xs = \tf12+\xe$. We then bound
    \[
        \ldint{\fr12+\xe} 2^{-s-1} \xz(s+1)\frac{\xz(s)}{\xz(2s)} \xG(s+j)\tau^{-s-j} \,ds \less_j X^{j+\fr12+\xe}
    \]
using arguments like in the proof of Lemma \ref{M:lem:PrincPsi1}, and thereby deduce that
	\<
		\label{M:eq:Psi0Asymp}
		(\Psi_0)_{(j)}(\rho e(\xb)) = \frac{j!}{4}\pfr{X}{1-2\pi iX\xb}^{j+1} + O_j(X^{j+\fr12+\xe}).
	\>
    Equation \eqref{M:eq:Momop} then follows at once by combining Lemmata \ref{M:lem:PrincPsi1} and equation \eqref{M:eq:Psi0Asymp}.

We now consider $\PsiD{j}(-\rho e(\xb))$ for $|\xb| \leq \xh$, again recalling that $\xh = X^{-1}(\logX)^{-1/4}$. First, equation \eqref{eq:Psi1Neg1} shows that
	\[
		(\Psi_1)_{(j)}(-\rho e(\xb)) = (\Psi_1)_{(j)}(-e^{-\tau}) = \sum_{k\equiv 1\mod2} \sum_{n=1}^\infty (-1)^n \mu(n) n^j k^{j-1} e^{-nk \tau},
	\]
and we now require a closed formula for $\Dir{(-1)^n\mu(n)}{s} = \sum_{n=1}^\infty (-1)^n \mu(n)n^{-s}$. If $ 2 \mid k$ then $\mu(2k) = 0$, and if $2 \nmid k$ then $\mu(2k) = - \mu(k)$, whereby
    \[
        \sum_{k=1}^\infty \frac{\mu(2k)}{(2k)^s} = -2^{-s} \sum_{2 \,\nmid\, k} \frac{\mu(k)}{k^s} = -2^{-s} \prod_{\substack{p\text{ prime} \\ p > 2}} (1-p^{-s}) = -\pfr{2^{-s}}{1-2^{-s}} \frac{1}{\xz(s)}.
    \]
We similarly see that $\sum_{n\equiv 1\mod2} (-1)^n\mu(n)n^{-s} = -[ (1-2^{-s})\xz(s) ]^{-1}$, and we deduce that
    \<
        \label{M:eq:DirNeg1}
        \Dir{(-1)^n\mu(n)}{s} = - \pfr{1+2^{-s}}{1-2^{-s}}\frac{1}{\xz(s)}.
    \>

The only significant difference between $\Dir{\mu}{s}$ and $\Dir{(-1)^n\mu(n)}{s}$ is the addition of simple poles at points $s = 2\pi ik/\log 2$ for $k \in \zz$. Because the arguments used for Lemma \ref{M:lem:PrincPsi1} only involve $s$ with $\xs \geq 1-c/\log T$, the analysis of $(\Psi_1)_{(j)}(-\xr e(\xb))$ is identical to that of $(\Psi_1)_{(j)}(+\xr e(\xb))$, yielding the bound
    \[
        (\Psi_1)_{(j)}(-\xr e(\xb)) \less X(\logX)^{-A},    
    \]
mutatis mutandis. Recalling the remarks at the end of section \ref{sec:PrincArcsI}, the functions $\Psi_0(z)$ and $(\Psi_0)_{(j)}(z)$ are even, so by equation \eqref{M:eq:Psi0Asymp} we immediately conclude that
    \[
        (\Psi_0)_{(j)}(-\rho e(\xb)) = \frac{j!}{4}\pfr{X}{1-2\pi iX\xb}^{j+1} + O_j(X^{j+\fr12+\xe}),
    \]
and formula \eqref{M:eq:Momop} for $z = -\rho e(\xb)$ follows.

We now turn our attention to equation \eqref{M:eq:Deriv},  following the arguments used in \cite{vaughan2008number}*{Lem.\ 2.3}. Arguing by induction, we find that for all $j \geq 1$ there exist integers $s(j,k)$ for which
    \<
        \label{eq:DiffOp}
        z^j \Psid{j}(z) = \sum_{k=1}^j s(j,k) \PsiD{j}(z).
    \>
Moreover, because $s(j,j)=1$ for all $j$, the leading term in the sum \eqref{eq:DiffOp} is $\PsiD{j}(z)$. Using the assumption that $|\xb| \leq X^{-1}(\logX)^{-1/4}$, we observe that 
    \[
        \xr e(\xb) = (1 + O(1/X)) (1+O(\xb)) = 1 + O(1/X),
    \]
whereby $(\pm\xr e(\xb))^j = (\pm1)^j + O_j(X^{-1})$. Using this with equations \eqref{M:eq:Momop} and \eqref{eq:DiffOp}, 
equation \eqref{M:eq:Deriv} is established for $z = \pm\xr e(\xb)$, and the proof of Proposition \ref{M:prop:Princ} is complete.

\begin{remark}
    We note that the $s(j,k)$ in \eqref{eq:DiffOp} are the \emph{Stirling numbers of the second kind} from elementary combinatorics. One also has a kind of dual relation to \eqref{eq:DiffOp}, in that one has
        \<
            \label{eq:DiffOp2}
            \PsiD{j}(z) = \sum_{k=1}^j S(j,k) z^j \Psid{j}(z),
        \>
    where the $S(j,k)$ are the \emph{Stirling numbers of the first kind}. Equations \eqref{eq:DiffOp} and \eqref{eq:DiffOp2} may both be proved using induction and the recursive relations satisfied by the $S(j,k)$ and $s(j,k)$. See, e.g., \cite{graham1994concrete}*{Exer.\ 6.13}.
\end{remark}

\section{Proof of Theorem \ref{M:thm:Relations}}
\label{sec:Relations}
We now examine how $\rho$ and $\xr_*$ (and thereby $X$ and $X_*$) should depend on $x$, recalling that for sufficiently large $x$ we want $\xr$ and $\xr_*$ to be solutions to the equations
    \[
        \PsiD{1}(\rho) = \xr\Psi'(\xr) = x \qqand \PsiD{1}(-\rho) = (-\xr)\Psi'(-\xr) = x,
    \] 
i.e., the equations \eqref{eq:sdp} and \eqref{eq:sdp*}, respectively.

\begin{lemma}
    \label{lem:sdpt}
    Let $\Psi(z) = \Psi(z,\mu)$. For all $x > x_0$ there exist unique $\rho=\xr(x)$ and $\rho_*=\xr_*(x)$ in $(0,1)$ satisfying equations \eqref{eq:sdp} and \eqref{eq:sdp*}, respectively. Moreover, as $x\to\infty$ one has $\xr(x)\to1$ and $\xr_*(x)\to1$.
\end{lemma}

\begin{proof}
    Setting $g(\rho) := \rho \Psi'(\rho)$, it is evident that $g(\xr)$ is continuous on $[0,1)$ and that $g(0) = 0$. Because Proposition \ref{M:prop:Princ} shows that $\lim_{\rho \to 1^-} g(\rho) = \infty$, a solution $\rho(x) \in (0,1)$ for \eqref{eq:sdp} is guaranteed by the intermediate value theorem. Moreover, since $g(\rho)$ is bounded on intervals $[0,\rho_0]\subset[0,1)$, it is necessary that $\rho(x) \to 1$ as $x \to \infty$. 
    
    We have $g'(\rho) = \Psi'(\rho)+\rho\Psi''(\rho)$, and equation \eqref{M:eq:Deriv} implies that $\lim_{\rho \to 1^-} g'(\rho) = \infty$, so there necessarily exists $\rho_0 \in (0,1)$ such that $g'(\rho) > 0$ for $\rho > \rho_0$. It follows that if $x > x_0 := \sup_{[0,\rho_0]} g(\rho)$, then $\rho(x) > \rho_0$, and the solution $\rho(x)$ to equation \eqref{eq:sdp} must be unique since $g(\rho)$ is monotonic on $(\rho_0,1)$. 

    The arguments concerning equation \eqref{eq:sdp*} and $\rho_*$ are identical.
\end{proof}

\begin{remark}
    Because $\rho = e^{-1/X}$ and $\xr(x)\to1$ as $x\to\infty$, the statements ``as $x \to \infty$'', ``as $X \to \infty$'', and ``as $\xr\to1$'' are equivalent, whereby asymptotic statements such as $X \sim 2\sqrt{x}$ may be made precise. For simplicity we omit further references to this technicality, and the dependencies of $\rho$ and $X$ on $x$ are understood in future discussions. The same is done for $\xr_*$, $X_*$, and $x$.
\end{remark}

Fix $A>0$. By Proposition \ref{M:prop:Princ}, for $X > X_0(A)$ one has
    \<
        \label{eq:xX}
        x = \tf14 X^2 \lf[1+O((\logX)^{-A})\rh],
    \>
whereby $\log x = 2 \logX + \O{1}$, and in particular $\log x \asymp \logX$. This implies that the term $O((\logX)^{-A})$ in equation \eqref{eq:xX} may be replaced with $O((\logx)^{-A})$, so that
    \<
        \label{eq:Xinx}
        X = 2\rt{x} \lh[ 1+O((\logx)^{-A}) \rh]^{-\fr12} = 2\rt{x}\lf[ 1+O((\logx)^{-A})\rh] 
    \>
and
    \<
        \label{eq:xLogrho}
        - x \log\rho = \frac{x}{X} = \tf12 \rt{x} \lf[ 1+O((\logx)^{-A}) \rh]^{-1} = \tf12 \rt{x}\lf[ 1+O((\logx)^{-A}) \rh].
    \>
Recalling from Proposition \ref{M:prop:Princ} that $\PsiD{j}(\rho) = \tf14 j!\: X^{j+1} + O(X^{j+1}(\logX)^{-A})$, for $j \geq 1$ we use equation \eqref{eq:Xinx} to deduce that
    \<
        \label{M:eq:Momopx}
        \begin{aligned}
        \PsiD{j}(\rho) 
            &= 2^{j-1} j!\: x^{\fr{j+1}2} \lf[ 1+O((\logx)^{-A}) \rh]^{j-1} = 2^{j-1} j!\: x^{\fr{j+1}2} \lf[ 1+O_j((\logx)^{-A}) \rh].
        \end{aligned}
    \>
Just as equation \eqref{M:eq:Deriv} is proved using equation \eqref{eq:DiffOp}, it follows from \eqref{M:eq:Momopx} that
    \<
        \label{M:eq:Derivx}
        (\pm1)^j\Psid{j}(\rho) = 2^{j-1} j!\: x^{\fr{j+1}{2}} \lf[ 1+O_j((\logX)^{-A}) \rh],
    \>
and we have established the relations of Theorem \ref{M:thm:Relations} for $\rho$ satisfying $\xr\Psi'(\rho)=x$.

Now considering the case of $\xr_*$ and $X_*$, again by Proposition \ref{M:prop:Princ} we have 
    \<
        \label{eq:xX*}
        x = \tf14 X_*^2 \lf[ 1+O((\log{X_*})^{-A}) \rh].
    \> 
Considering the identical forms of equations \eqref{eq:xX} and \eqref{eq:xX*}, the analogues of equations \eqref{eq:Xinx}--\eqref{M:eq:Derivx} for $-x \log(\rho_*)$, $\PsiD{j}(-\rho_*)$, etc.\! follow mutatis mutandis, completing the proof of Theorem \ref{M:thm:Relations}.

\section{Arc transference and \tops{$*$}{*}-arcs}
\label{sec:transference}

Having considered $\Psi(\xr e(\xa))$ for different $\xa \in [0,1)$ and for $\xr$ and $\xr_*$ as functions of $x$, we now return to the integral formula for $p(x,\mu)$, namely
    \<
        \label{eq:ints}
        \dint_{|z|=r} \Phi(z)z^{-x-1} \,dz = r^{-x} \int_0^1 \exp\Psi(r e(\xa)) e(-x\xa) \,d\xa \qquad (0<r<1), 
    \>
and consider the different contributions made to \eqref{eq:ints} as $\xa$ runs over $[0,1)$. As mentioned in section \ref{sec:prelim}, ultimately only those $\xa$ in small neighborhoods of $0$, $1$, and $1/2$ significantly contribute to this integral.

In what follows let $x>0$ be sufficiently large. 
In analyzing integral \eqref{eq:ints} when $\xa$ is near $0$ and $1$, control of our error terms requires that we take the radius $r$ of integration to be the solution $\xr$ of \eqref{eq:sdp}. On the other hand, controlling our error terms in \eqref{eq:ints} for $\xa$ near $1/2$ similarly requires that $r$ be the solution $\xr_*$ of \eqref{eq:sdp*}. As we expect for almost all $x$ that $\rho(x)$ and $\rho_*(x)$ are not equal, we must make a ``compromise'' by modifying the original circle $|z|=r$ of integration in \eqref{eq:ints}.

As $\Phi(z)$ is analytic for $|z|<1$, we split the interval $[0,1)$ into the sets 
    \<
        \label{eq:fAfA*}
        \fA := [0,\tf14)\cup[\tf34,1) \qqand \fA_* := [\tf14,\tf34)
    \>
and change the path of integration on the left side of \eqref{eq:ints} to be the positively oriented path consisting of the two semicircles 
    \[
        \rho e(\fA) := \{ \rho e(\xa) : \xa \in \fA \} \qqand \rho_* e(\fA_*) := \{ \rho_* e(\xa) : \xa \in \fA_* \}
    \]
and the two connecting line segments $[i\rho,i\rho_*]$ and $[-i\rho_*,-i\rho]$, as shown in Figure \ref{fig:Contour}.

{
\begin{figure}[!ht]
\centering
{\scalebox{0.95}{
\begin{tikzpicture}
    \def\gap{0.2}
    \def\tzR{2.35} 
    \def\tzr{1.80} 
    \def\tzudr{3} 
    
    \draw[line width = 0.5pt] (0,0) circle (\tzudr);
    \draw[dashed, line width = 0.5pt] (0,0) circle (\tzR);
    \draw[dashed, line width = 0.5pt] (0,0) circle (\tzr);
    
    \draw [help lines,-] (-1.2*\tzudr, 0)--(1.2*\tzudr, 0);
    \draw [help lines,-] (0, -1.2*\tzudr)--(0, 1.2*\tzudr);
    
    \draw[line width=1pt,
        decoration = {markings, mark = at position 0.6 with {\arrow[line width=1.2pt]{>}}},
        postaction = {decorate}] 
        (90:\tzR) arc (90:270:\tzR);
    
    \draw[line width=1pt, 
        decoration = {markings, mark = at position 0.6 with {\arrow[line width=1.2pt]{>}}}, postaction = {decorate}] 
        (-90:\tzR) -- (-90:\tzr);
    
    \draw[line width=1pt, 
        decoration = {markings, mark = at position 0.6 with {\arrow[line width=1.2pt]{>}}}, postaction = {decorate}] 
        (-90:\tzr) arc (-90:90:\tzr);
    
    \draw[line width=1pt, 
        decoration = {markings, mark = at position 0.6 with {\arrow[line width=1.2pt]{>}}}, postaction = {decorate}] 
        (90:\tzr) -- (90:\tzR);
    
    \draw[very thin, dashed] (0,0)--(125:\tzR);
    \node[below] at (130:0.6*\tzR) {$\rho_*$};
    
    \draw[very thin, dashed] (0,0)--(345:\tzr);
    \node[below] at (345:0.6*\tzr) {$\rho$};
    
    \node[below right] at (\tzudr,0) {$1$};
    \node[below left] at (-\tzudr,0) {$-1$};
    \node[above right] at (0, \tzudr) {$i$};
    \node[below right] at (0, -\tzudr) {$-i$};
\end{tikzpicture}
}}
\caption{The new contour of integration. We note that the figure draws $\rho_* > \rho$ for illustrative purposes only---it is possible that $\rho_* \leq \rho$.}
\label{fig:Contour}
\end{figure}
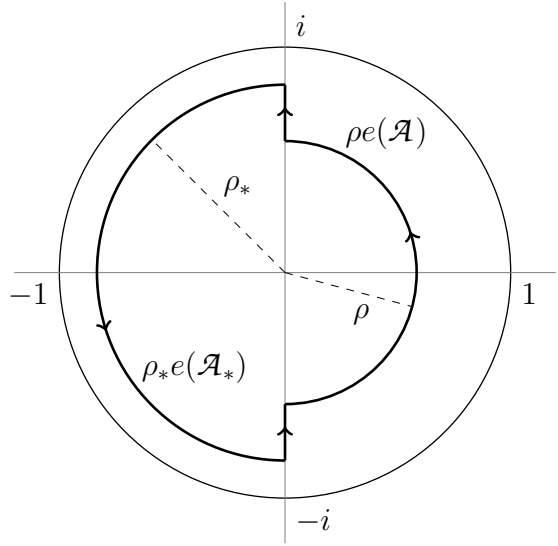
}

Using this new contour of integration in \eqref{eq:ints} we have
    \<
        \label{eq:pxmuNewInt}
        p(x,\mu) = \dint_{\rho e(\fA)} \Phi(z) z^{-x-1}\,dz + \dint_{\rho_* e(\fA_*)} \Phi(z) z^{-x-1} \, dz + \fT(\rho,\rho_*),
    \>
where
    \<
        \fT(\xr,\xr_*) := \dint_{i\xr}^{i\xr_*} \Phi(z)z^{-n-1} \,dz + \dint_{-i\xr_*}^{-i\xr} \Phi(z)z^{-n-1} \,dz,
    \>
and we note that
    \[
        \fT(\rho,\rho_*) = 2 \Re \lf[ \dint_{i\rho}^{i\rho_*} \Phi(z)z^{-n-1} \,dz \rh] = 2\Re\lf[ \frac{i^{-n-1}}{2\pi} \int_{\rho}^{\rho_*} \Phi(it)t^{-n-1} \,dt \rh].
    \]

\begin{remark}
    Before proceeding further we must address a technical issue introduced by this ``arc transference'' in equation \eqref{eq:pxmuNewInt}. Because $\xr(x)$ and $\xr_*(x)$ are unlikely to be equal, so then are $X$ and $X_*$ unlikely to be equal. Thus, the division of $[0,1)$ into major, minor, and principal arcs using $X$ and $Q$ (which is a power of $X$) is now problematic. In particular, we must define ``$*$-arcs'' using $X_*$ and $Q_*$ via
        \[
            \fM_*(a/q) := \{ \xa \in [0,1) : |\xa-a/q| \leq Q_*/(qX_*) \},    
        \]
    and consider $*$-versions of the major and minor arcs, say $\fM(X_*,Q_*)$ and $\fm(X_*,Q_*)$, noting that we have already defined $\fP_*$.  

    Fortunately, this technicality does not create much trouble. In particular, recalling the note at the end of section \ref{sec:prelim}, the proofs of Propositions \ref{M:prop:Minor}, \ref{M:prop:Major}, and \ref{M:prop:Princ} do not require that $\xr$ be a solution to one of the equations \eqref{eq:sdp} and \eqref{eq:sdp*}. Indeed, said propositions hold by identical proofs if one replaces $\xr$, $X$, and $Q$ with $\xr_*$, $X_*$, and $Q_*$, respectively. That this replacement is valid is perhaps made intuitively clear by recalling that, per Theorem \ref{M:thm:Relations}, both $X \sim 2\sqrt{x}$ and $X_* \sim 2\sqrt{x}$.
\end{remark}

We now establish that this ``arc transference'' made in equation \eqref{eq:pxmuNewInt} only adds a small error term to our formula for $p(x,\mu)$; that is, small  relative to the main terms of \eqref{eq:pxmuNewInt}, which are of order $\exp(\sqrt{x})$. For convenience we recall from Theorem \ref{M:thm:Relations} and the previous section that
    \begin{alignat}{2}
        \label{M:eq:XinxRecall}
        -x \log \rho &= \tf12 \rt{x}\brk{1+O((\logx)^{-A})},
            \qquad X &&= 2\rt{x}\brk{1+O((\logx)^{-A})}, \\
        \label{M:eq:X*inxRecall}
        -x \log \rho_* &= \tf12 \rt{x}\brk{1+O((\logx)^{-A})},
            \qquad X_* &&= 2\rt{x}\brk{1+O((\logx)^{-A})}.
    \end{alignat}

\begin{lemma}
    \label{lem:Transfer}
    For $x > x_0$ let $\rho$ and $\rho_*$ be the unique solutions to equations \eqref{eq:sdp} and \eqref{eq:sdp*}, respectively. For all $\xe > 0$, as $x\to\infty$ one has
        \[
            \fT(\rho,\rho_*) \less \exp((\tf12+\xe)\sqrt{x}).
        \]
\end{lemma}

\begin{proof}
    Fix $A,\xe>0$ and let $x>x_0(A)$. It is evident that $\fT(\rho,\rho_*) = -\fT(\rho_*,\rho)$, so there is no loss of generality in assuming that $\rho\leq\rho_*$. First, using \eqref{M:eq:XinxRecall} and \eqref{M:eq:X*inxRecall} we find that
        \begin{align}
            \label{eq:TIntBB}
            \begin{aligned}
            & \int_{\rho}^{\rho_*} \frac{dt}{t^{x+1}} 
            = \frac{\rho^{-x}-\rho_*^{-x}}{x}
            = \fr{1}{x}\lf[ e^{\fr12\rt{x} + o(\rt{x})} - e^{\fr12\rt{x} + o(\rt{x})} \rh]
            \less e^{(\fr12+\xe)\rt{x}}.
            \end{aligned}
        \end{align}
    
    For $\rho \leq t \leq \rho_*$ let $t = \exp(-1/X_t)$, and assume that $x$ (and therefore $\rho$) is large enough that $X_t (\logX_t)^{-A}$ is increasing in $t$ for $t \geq \rho$. Because $\Vmaj(4) = 0$ in Proposition \ref{M:prop:Major}, for sufficiently large $x$ it holds that
        \<
            \label{eq:PsiitBB}
            \Re \Psi(it) \less X_t (\logX_t)^{-A} \less X_* (\logX_*)^{-A} \less \sqrt{x} (\logx)^{-A} \leq \xe\sqrt{x} \qquad 
        \>
    for $\rho \leq t \leq \rho_*$. Using inequalities \eqref{eq:TIntBB} and \eqref{eq:PsiitBB} we have
        \[
            \lf| \dint_{i\rho}^{i\rho_*} \Phi(z)z^{-x-1} \,dz \rh| 
            \less e^{\xe\rt{x}} \int_{\rho}^{\rho_*}\frac{dt}{t^{x+1}} \less e^{(\fr12 + 2\xe)\rt{x}},
        \]
    and the result follows by repeating our arguments with $\xe/2$ in place of $\xe$.
\end{proof}

\begin{corollary}
    \label{M:cor:transfer}
    For $x > x_0$ let $\rho$ and $\rho_*$ be the unique solutions to equations \eqref{eq:sdp} and \eqref{eq:sdp*}, respectively. For all $\xe > 0$, as $x\to\infty$ one has
        \[
            \begin{aligned}
                p(x,\mu) 
                &= \rho^{-x} \int_{\fA} \Phi(\rho e(\xa)) e(-x\xa) \, d\xa + \rho_*^{-x} \int_{\fA_*} \Phi(\rho_* e(\xa)) e(-x\xa)\,d\xa + O\big( e^{(\fr12+\xe)\sqrt{x}} \big).
            \end{aligned}
        \]
\end{corollary}

\section{The contributions of nonprincipal arcs}
\label{sec:nonPrc}

We are now ready to bound the contributions of all nonprincipal arcs to the integrals 
    \[
        \rho^{-x} \int_{\fA} \Phi(\rho e(\xa))e(-x\xa) \,d\xa \qqand \rho_*^{-x} \int_{\fA_*} \Phi(\rho_* e(\xa))e(-x\xa) \,d\xa,
    \]
recalling that 
    \[
        \fA := [0,\tf14)\cup[\tf34,1) \qqand \fA_* := [\tf14,\tf34)  
    \]
and that
    \begin{align*}
        \fP &= \{ \xa\in[0,1) : \|\xa\| \leq X^{-1}(\logX)^{-\nfr14} \}, \\
        \fP_* &= \{ \xa\in[0,1) : |\xa-\tf12| \leq X_*^{-1}(\logX_*)^{-\nfr14} \}.
    \end{align*}

Our primary tools toward this end are the following inequalities on $\Re\Psi(z)$.

\begin{lemma}
    \label{M:lem:NonPrincIneq}
    Let $X > X_0$. For all $\xa \in \fA\setminus\fP$ one has
        \<
            \label{M:eq:PsiReBB}
            \Re \Psi(\rho e(\xa)) \leq \lf( 1-\fr{1}{\logX} \rh) \Psi(\rho),
        \>
    and for all $\xa \in \fA_*\setminus\fP_*$ one has
        \<
            \label{M:eq:PsiReBB*}
            \Re \Psi(\rho_* e(\xa)) \leq \lf( 1-\fr{1}{\logX} \rh) \Psi(-\rho_*).
        \>
\end{lemma}

\begin{remark}
    In the following proof we require that $A>1$ in contrast to our usual allowance that $A>0$. This restriction does not affect any of our previous results since said results allowed $A>0$ to be arbitrary.
\end{remark}

\begin{proof}
    Fix $A>1$, let $X>X_0(A)$, and let $Q=X^{2/5}$. Suppose that $\xa \in \fA$ can be written as $\xa = a/q+\xb$ with $0 \leq a \leq q \leq Q$, $(a,q)=1$, and $|\xb|\leq Q/(qX)$. For real $\xb$ we define
        \[
            \Delta = (1+4\pi^2X^2\xb^2)^{-\fr12} \qqand \vphi = \arctan(2\pi X\xb), 
        \]
    using the branch of $\arctan$ with values in $(-\fr{\pi}{2},\fr{\pi}{2})$, so that $(1-2\pi iX\xb)^{-1} = \Delta e^{i\vphi}$ and Proposition \ref{M:prop:Major} states that
        \<
            \label{eq:PsiMajorRedux}
            \Psi(\xr e(\xa)) = V(q)X\xD e^{i\vphi} + O(X(\logX)^{-A}).
        \>
    As $\Vmaj(1)=\tf14$, we have
        \<
            \label{M:eq:Psirho}
            \Psi(\xr) = \tf14 X + O(X(\logX)^{-A}).
        \>
    
    Since $\Delta \leq 1$ and $|\Vmaj(q)| \leq \tf{1}{16}$ when $q>2$, it follows at once from \eqref{eq:PsiMajorRedux} and \eqref{M:eq:Psirho} that inequality \eqref{M:eq:PsiReBB} holds for all $\xa$ in sets $\fA \cap \fM(a/q)$ with $3 \leq q \leq Q$. Recalling from Proposition \ref{M:prop:Minor} that for $\xa \in \fA \cap \fm(X,Q)$ one has $\Psi(\rho e(\xa)) \less X (\logX)^{-A}$ already, then \emph{a fortiori} inequality \eqref{M:eq:PsiReBB} holds for these $\xa$ as well. 

    It thus remains to consider only those $\xa$ in the major arcs $\fM(\fr01)$ and $\fM(\fr11)$ (since $\fM(\fr12)$ is not included in $\fA$) that are \emph{not} in the principal arc $\fP$. Considering the case of $\fM(\fr01)$, let $\xa = \xb$ with $X^{-1}(\logX)^{-1/4} < \xb \leq Q/X$. Here we find that
        \[
            \Delta \leq \big( 1+4\pi^2(\logX)^{-\fr12} \big)^{-\fr12} = 1-2\pi^2(\logX)^{-\fr12} + O((\logX)^{-1}) \leq 1-(\logX)^{-\fr12}
        \]
    for sufficiently large $X$, whereby equation \eqref{eq:PsiMajorRedux} shows that
        \<
            \label{eq:PsiReNonP}
            \Re \Psi(\xr e(\xb)) \leq \tf14(1-(\logX)^{-\fr12})X + O(X(\logX)^{-A}).
        \>
    Since the main term in $(1-(\logX)^{-1})\Psi(\xr)$ is $\tf14(1-(\logX)^{-1})X$ and we have $A>1$, we may ``spend'' the extra saved $(\logX)^{-1/2}$ in \eqref{eq:PsiReNonP} to overcome any differences between the error terms in \eqref{M:eq:Psirho} and \eqref{eq:PsiReNonP}, and thereby conclude that inequality \eqref{M:eq:PsiReBB} holds for all $\xb \in \fM(\tf01) \setminus \fP$. 
    
    A similar argument shows that inequality \eqref{M:eq:PsiReBB} holds for $\xa \in \fM(\tf11) \setminus \fP$, so the lemma's first assertion is established. Since $V(2)$ is also equal to $\tf14$, Proposition \ref{M:prop:Major} similarly implies that 
        \[
            \Psi(-\rho_*) = \tf14 X_* + O(X_*(\logX_*)^{-A}).
        \]
    Using this in place of \eqref{M:eq:Psirho}, identical arguments to those above show that inequality \eqref{M:eq:PsiReBB*} holds for $\xa \in \fA_* \setminus \fP_*$, and the proof of the lemma is completed.
\end{proof}

Using the results of Lemma \ref{M:lem:NonPrincIneq} we now take our first major step toward establishing the asymptotic formula \eqref{M:eq:pxmuForm} for $p(x,\mu)$.

\begin{proposition}
    \label{M:prop:pxmuReduc}
    For all $x > x_0$ let $\xr$ and $\xr_*$ be the unique solutions to equations \eqref{eq:sdp} and \eqref{eq:sdp*}, respectively. For fixed $B>0$, as $x\to\infty$ one has
        \<
            \label{eq:pxmuIntP}
            \begin{aligned}
                p(x,\mu) &= \rho^{-x} \int_{\fP} \Phi(\rho e(\xa)) e(-x\xa) \, d\xa + \rho_*^{-x} \int_{\fP_*} \Phi(\rho_*e(\xa)) e(-x\xa) \,d\xa \\
                & \qquad + O(\rho^{-x}\Phi(\rho)x^{-B}) + O(\rho_*^{-x}\Phi(-\rho_*)x^{-B}) + O(e^{(\fr12+\xe)\sqrt{x}}).
            \end{aligned}
        \>
\end{proposition}

\begin{proof}
    Fix $B>0$ and let $x>x_0(B)$. When $\xa \in \fA \setminus \fP$ Lemma \ref{M:lem:NonPrincIneq} implies that
        \<
            \label{eq:expPsiBB}
            \Phi(\xr e(\xa)) = \exp \Psi(\rho e(\xa)) \less \expp{ \Psi(\rho) - \frac{\Psi(\rho)}{\logX} },
        \>    
    and since $\Psi(\xr) \asymp X \asymp x^{\nfr14}$ and $\logX \asymp \logx$ by \eqref{M:eq:XinxRecall}, we have
        \<
            \label{eq:expPsiBB2}
            \expp{-\frac{\Psi(\rho)}{\logX}} \less \expp{-\frac{x^{\nfr12}}{\log x}} \less \exp(-x^{\nfr14}) \less x^{-B}.  
        \>
    By \eqref{eq:expPsiBB} and \eqref{eq:expPsiBB2} then, we have
        \[
            \rho^{-x} \int_{\fA\setminus\fP} \Phi(\rho e(\xa)) e(-x\xa)\, d\xa \less \rho^{-x}e^{\Psi(\rho)} x^{-B} = \rho^{-x}\Phi(\rho)x^{-B},
        \]
    and similarly for the integral over $\fA_* \setminus \fP_*$ (by way of inequality \eqref{M:eq:PsiReBB*}), and the result follows at once from Corollary \ref{M:cor:transfer}.
\end{proof}

\section{The proof of Theorem \ref{M:thm:asymp}}
\label{sec:MAsymp}

To establish our asymptotic formula for $p(x,\mu)$ in Theorem \ref{M:thm:asymp}, it only remains to treat the principal-arc integrals in equation \eqref{eq:pxmuIntP}. We recall that we identify 
    \[
        \{e(\xa):\xa\in\fP\} \qqand \{e(\xa):\xa\in\fP_*\}
    \]
with
    \[
        \{ e(\xb) : \text{$\xb \in \rr$ and $|\xb|\leq \xh$} \} \qqand \{ -e(\xb) : \text{$\xb\in\rr$ and $|\xb|\leq\xh_*$} \},
    \]
respectively, where
    \[
        \xh = X^{-1}(\logX)^{-\nfr14} \qqand \xh_* = X_*^{-1}(\logX_*)^{-\nfr14}.
    \]

When $f$ is nonnegative it is clear that $|\Psi(z,f)| \leq \Psi(|z|,f)$ for $|z|<1$, and similarly for all derivatives of $\Psi(z,f)$. We recover a similar bound for $\Psi(z,\mu)$ in the following lemma.

\begin{lemma} 
    \label{lem:PsiDerivBB}
    Fix $A>0$. There exists a constant $X_A > 0$ such that for all $0 \leq j \leq 3$, $X > X_A$, and $|\xb| \leq \xh$ one has
        \<
            \label{eq:PsiDerivBB}
            |\PsiD{j}(\rho e(\xb))| \leq \rt{2} \PsiD{j}(\rho).
        \>
\end{lemma}

\begin{proof}
    For $0 \leq j \leq 3$ let $c_j$ denote the implicit constants in equation \eqref{M:eq:Momop}. By said equation, for $X>X_0(A)$ we have
        \[
            \frac{|\PsiD{j}(\rho e(\xb))|}{\PsiD{j}(\rho)}= |1-2\pi iX\xb|^{-j-1}\lf|\frac{1+O_j(\log^{-A}X)}{1+O_j(\log^{-A}X)} \rh| \leq \frac{1+c_j\log^{-A}X}{1-c_j\log^{-A}X}.
        \]
    The latter fraction here is at most $\sqrt{2}$ for $X > X_0(j,A)$, and inequality \eqref{eq:PsiDerivBB} follows by setting $X_A := \max_{0\leq j\leq 3} X_0(j,A)$.
\end{proof}

We first consider the integral in \eqref{eq:pxmuIntP} over $\fP$ using arguments from \cite{vaughan2008number}*{sec.\ 4}. 

\begin{proposition}
    \label{M:prop:MainTerm}
    For $x>x_0$ let $\rho$ be the unique solution to equation \eqref{eq:sdp}. As $x\to\infty$, one has
        \<
            \label{M:eq:PrincIntegral0}
            \rho^{-x} \int_{\fP} \Phi(\rho e(\xa)) e(-x\xa)\,d\xa = \frac{\rho^{-x}\Phi(\rho)}{\sqrt{2\pi \PsiD{2}(\rho)}} \lf[ 1+O(x^{-\nfr15}) \rh].
        \>
\end{proposition}

\begin{remark}
    In the following proof, quantities $c_1, c_2, \ldots$ are positive constants.
\end{remark}

\begin{proof}
    For $|\xb| \leq \xh$ let $\Psi(\rho e(\xb))=R(\xb)+iI(\xb)$ for real-valued functions $R(\xb)$ and $I(\xb)$, and consider the third-order Taylor expansion with remainder:
        \begin{align*}
            R(\xb) &= R(0) + \xb R'(0) + \tf12 \xb^2 R''(0) + \tf16 \xb^3 R'''(c_R \xb), \\
            I(\xb) &= I(0) + \xb I'(0) + \tf12 \xb^2 I''(0) + \tf16 \xb^3 I'''(c_I \xb),
        \end{align*}
    where $c_R$ and $c_I$ are in $(0,1)$ and depend on $\xb$. One may easily see that
        \<
            \label{eq:RIDeriv}
            R^{(j)}\.(\xb) + iI^{(j)}\.(\xb) = \partial_\xb^{\,j}\hspace{-0.1em}\big[\Psi(\rho e(\xb))\big] = (2 \pi i)^j \PsiD{j}(\rho e(\xb))
        \>
    for all $j \geq 1$, and in particular that
        \begin{subequations}
        \begin{align}
            R'(\xb) + i I'(\xb) &= 2 \pi i \PsiD{1}(\rho e(\xb)), \\
            R''(\xb) + i I''(\xb) &= - 4 \pi^2 \PsiD{2}(\rho e(\xb)), \\
            \label{eq:RIDeriv2}
            R'''(\xb) + i I'''(\xb) &= -8 \pi^3 i \PsiD{3}(\rho e(\xb)).
        \end{align}
        \end{subequations}
    Applying Lemma \ref{lem:PsiDerivBB} to equation \eqref{eq:RIDeriv2}, we have
        \[
            |R'''(\xb)| \leq 8 \pi^3 \sqrt{2} \PsiD{3}(\rho), 
        \]
    and likewise for $|I'''(\xb)|$. With $c_0 := 8 \pi^3 \sqrt{2} \PsiD{3}(\rho)$, we implicitly define $w(\xb)$ via
        \<
            \label{eq:w}
            \xb^3 \big[R'''(c_R \xb) + i I'''(c_I \xb) \big] = 2 c_0 |\xb|^3  w(\xb)
        \>
    and see that $|w(\xb)| \leq 1$. Combining equations \eqref{eq:RIDeriv}--\eqref{eq:w}, we deduce that
        \<
            \label{eq:PsiTaylor}
            \Psi(\rho e(\xb)) = 
            \Psi(\rho) + 2\pi i\xb\PsiD{1}(\rho) - 2\pi^2\xb^2 \PsiD{2}(\rho) + \tf83 \pi^3 w(\xb)|\xb|^3\PsiD{3}(\rho).
        \>

    Recalling that $x = \PsiD{1}(\rho)$ we have $e(-x\xb) = \exp(-2\pi i\xb\PsiD{1}(\rho))$, and we thus write
        \<
            \label{eq:IntxQ}
            \xr^{-x} \int_{\fP} \exp\.\.\hspace{-0.04em}\big[\Psi(\rho e(\xa))\big] e(-x\xa) \,d\xa
                = \rho^{-x} \int_{-\xh}^{\xh} \exp\xQ(\rho,\xb) \,d\xb,
        \>
    where
        \<
            \label{eq:xQ}
            \xQ(\rho,\xb) :=
            \Psi(\rho) - 2 \pi^2 \xb^2 \PsiD{2}(\rho) + \tf83 \pi^3 w(\xb) |\xb|^3 \PsiD{3}(\rho).
        \>
    By Theorem \ref{M:thm:Relations} we have $\PsiD{2}(\rho) \great X^3$ and $\PsiD{3}(\rho) \less X^4$, so that when $X$ is sufficiently large and $|\xb| \leq \eta = X^{-1}(\logX)^{-1/4}$ we have
        \[
            \Re \xQ(\rho,\xb) \leq \Psi(\rho) - 1 \pi^2 \xb^2 \PsiD{2}(\rho).
        \]
    For those $\xb$ such that $|\xb| \geq X^{-3/2} (\logX)^{3/2}$ as well, we further deduce that
        \<
            \label{eq:xQReBB}
            \Re \xQ(\rho,\xb) \leq \Psi(\rho) - c_1(\logX)^3.
        \>
    Fixing $B>1$ and setting 
        \<
            \xd := X^{-\nfr32} (\logX)^{\nfr32},
        \>
    we use \eqref{eq:xQReBB} and repeat the arguments used for Proposition \ref{M:prop:pxmuReduc} to conclude that
        \<
            \label{M:eq:IntExtra}
            \xr^{-x} \int_{\xd\leq|\xb|\leq\xh} \exp\xQ(\xr,\xb) \,d\xb \less \xr^{-x}\Phi(\xr)x^{-B}.  
        \>

    It remains to consider \eqref{M:eq:PrincIntegral0} when $|\xb| < \xd$. For these $\xb$, Theorem \ref{M:thm:Relations} implies that
        \[
            \begin{aligned}
            &\tf83 \pi^3 |\xb|^3 w(\xb) \PsiD{3}(\rho) \less X^{-\fr92} (\logX)^{\fr92} X^4 \sim c_2 x^{-\fr14} \lf[ (\logx)^{\fr92} + \O{1} \rh] \less x^{-\fr15},
            \end{aligned}
        \]
    and we deduce that
        \<
            \rho^{-x} \int_{-\xd}^\xd \exp\xQ(\rho,\xb) \,d\xb = \rho^{-x}e^{\Psi(\rho)}\lf[1 + O(x^{-\fr15}) \rh] \int_{-\xd}^{\xd} \expp{ - 2\pi^2 \PsiD{2}(\rho) \xb^2 } \,d\xb.
        \>
    We have
        \[
            \int_{-\xd}^{\xd} \expp{ -2\pi^2 \PsiD{2}(\rho) \xb^2 } \,d\xb = \frac{1}{\rt{2\pi\PsiD{2}(\rho)}} + \O{\frac{\exp(-\xd^2\PsiD{2}(\rho))}{\xd\PsiD{2}(\rho)}},  
        \]
    and since $\PsiD{2}(\rho) \asymp X^3 \asymp x^{3/2}$ our error term here is 
        \[
            \less \frac{\exp(-c_3(\logX)^3)}{X^{3/2}(\logX)^{3/2}} \less \exp(-c_4 (\logx)^3) \less x^{-B}.
        \]
    In total then, we find that
        \[
            \rho^{-x} \int_{-\xd}^\xd \exp\xQ(\rho,\xb) \,d\xb
            = \rho^{-x}\Phi(\rho)\lf[ 1 + O(x^{-\nfr15}) \rh] \lf[ \frac{1}{\rt{2 \pi \PsiD{2}(\rho)}} + O(x^{-B}) \rh].
        \]
    Using the assumption that $B>1$ we deduce that
        \<
            \label{M:eq:IntDelta}
            \rho^{-x} \int_{-\xd}^\xd \exp\xQ(\rho,\xb) \,d\xb = \frac{\rho^{-x}\Phi(\rho)}{\rt{2 \pi \PsiD{2}(\rho)}} \lf[1+O(x^{-\nfr15})\rh],
        \>
    and the result follows from adding equations \eqref{M:eq:IntExtra} and \eqref{M:eq:IntDelta}, again due to the facts that $\PsiD{2}(\rho) \asymp x^{3/2}$ and $B > 1$.
\end{proof}

As is par for the course, the analysis of $\Psi(\rho_* e(\xa))$ for $\xa$ in $\fP_*$ is largely identical to the analysis for $\xa$ in $\fP$. Thus, let $\rho_*$ be the unique solution to \eqref{eq:sdp*}, and for $\xa \in \fP_*$ write $\xa = \tf12+\xb$ and $\rho_* e(\xa) = -\rho_* e(\xb)$ with $|\xb| \leq \xh_* = X_*^{-1}(\logX_*)^{-1/4}$. 

Like in Lemma \ref{lem:PsiDerivBB}, we easily deduce that for fixed $A>0$ there is some $X_A>0$ such that: For $0\leq j \leq 3$, $X_*>X_A$, and $\xa \in \fP_*$ one has
    \[
        |\PsiD{j}(\rho_* e(\xa))| = |\PsiD{j}(-\xr_* e(\xb))| \leq \rt{2}\PsiD{j}(-\rho_*).
    \]
As $x = \PsiD{1}(-\rho_*)$ now, this time we have 
    \[
        e(-x\xa) = e(-\tf12x)e(-x\xb) = e(-\tf12 x) \exp(-2\pi i \xb \PsiD{1}(-\rho_*)).
    \]
Letting $\Psi(-\rho_* e(\xb)) = R(\xb)+iI(\xb)$ for some new $R$ and $I$, we again compute that 
    \[
        \Psi(-\rho_* e(\xb)) = \Psi(-\rho_*) + 2 \pi i \xb \PsiD{1}(-\rho_*) - 2 \pi^2 \xb^2 \PsiD{2}(-\rho_*) 
            + \tf83 \pi^3 w_*(\xb) |\xb|^3 \PsiD{3}(-\rho_*),
    \]
cf.\ \eqref{eq:PsiTaylor}, where $w_*(\xb)$ is defined as $w(\xb)$ was in equation \eqref{eq:w}. It follows that
    \<
        \label{M:eq:IntxQ*}
        \xr_*^{-x} \int_{\fP_*} \exp\!\big[\Psi(\rho_* e(\xa))\big] e(-x\xa) \,d\xa 
        =  e(-\tf12 x)\xr_*^{-x} \int_{-\xh_*}^{\xh_*} \exp\xQ_*(\rho_*,\xb) \,d\xb,
    \>
where
    \<
        \label{eq:xQ*}
        \xQ_*(\rho_*,\xb) := \Psi(-\rho_*) - 2 \pi^2 \xb^2 \PsiD{2}(-\rho_*) + \tf83 \pi^3 w_*(\xb) |\xb|^3 \PsiD{3}(-\rho_*).
    \>

Considering the overwhelming similarities between equations \eqref{M:eq:IntxQ*} and \eqref{eq:xQ*} and equations \eqref{eq:IntxQ} and \eqref{eq:xQ}, respectively, and the symmetric relations for $\PsiD{j}(\pm\xr)$ in Theorem \ref{M:thm:Relations}, we simply repeat the proof of Proposition \ref{M:prop:MainTerm} to conclude that
    \<
        \label{M:eq:MainTerm*}
        \rho_*^{-x} \int_{\fP_*} \Phi(\rho_* e(\xa)) e(-x\xa)\,d\xa = e(-\tf12 x) \frac{\rho_*^{-x}\Phi(-\rho_*)}{\sqrt{2\pi\PsiD{2}(-\rho_*)}} \lf[ 1+O(x^{-\nfr15}) \rh].
    \>

\begin{remark}
   To maintain consistency between the formulae for $p(x,f)$ and $p(n,f)$, in the remainder of this paper we write $e(-\tf12 x) = (-1)^{-x}$ and understand all future occurrences of $(-1)^{-x}$ to use this identification.
\end{remark}

Combining Propositions \ref{M:prop:pxmuReduc} and \ref{M:prop:MainTerm} and equation \eqref{M:eq:MainTerm*}, as $x\to\infty$ one has
    \<
        \label{M:eq:pxmuAsymp}
        \begin{aligned}
            p(x,\mu) &= \frac{\rho^{-x}\Phi(\rho)}{\rt{2\pi\PsiD{2}(\rho)}} \lf[ 1+O(x^{-\nfr15}) \rh]
            + (-1)^{-x} \frac{\rho_*^{-x}\Phi(-\rho_*)}{\rt{2\pi\PsiD{2}(-\rho_*)}} \lf[ 1+O(x^{-\nfr15}) \rh] \\
            & \qquad + O(e^{(\fr12+\xe)\rt{x}}),
        \end{aligned}
    \>
which is nearly equation \eqref{M:eq:pxmuForm}. Using Theorem \ref{M:thm:Relations} one finds that
    \[
        \frac{\rho^{-x}\Phi(\rho)}{\sqrt{2\pi\PsiD{2}(\rho)}} = \exp\!\Big(\Psi(\rho) - x \log(\rho) - \tf12 \log(2\pi\PsiD{2}(\rho))\Big) = \expp{\sqrt{x} + o(\sqrt{x})}, 
    \]
whereby we simplify equation \eqref{M:eq:pxmuAsymp} by writing 
    \[
        \frac{\rho^{-x}\Phi(\rho)}{\rt{2\pi\PsiD{2}(\rho)}} \lf[ 1+O\hspace{-0.05em}\big(x^{-\nfr15}\big) \rh] + O\hspace{-0.05em}\big(e^{(\fr12+\xe)\rt{x}}\big) = \frac{\rho^{-x}\Phi(\rho)}{\rt{2\pi\PsiD{2}(\rho)}} \lf[ 1+O\hspace{-0.05em}\big(x^{-\nfr15}\big) \rh]
    \] 
to finally yield the formula
    \<
        \label{M:eq:pxmuAsymp2}
        p(x,\mu) = \fr{\xr^{-x}\Phi(\xr)}{\rt{2\pi\PsiD{2}(\xr)}} \lf[ 1+O(x^{-\nfr15}) \rh]
        + (-1)^{-x} \fr{\xr_*^{-x}\Phi(-\xr_*)}{\rt{2\pi\PsiD{2}(-\xr_*)}} \lf[ 1+O(x^{-\nfr15}) \rh],
    \>
and the proof of Theorem \ref{M:thm:asymp} is complete.


\section{Formulae for \tops{$p(x,\xl)$}{p(x,λ)}}
\label{sec:Liouville}

Given the functions' strong connection, results for one of M\"obius' $\mu$ and Liouville's $\xl$ functions are often easily translated into results for the other. Our formulae for $p(x,\mu)$ and $p(x,\xl)$ conform to this pattern. Many of the proofs of our lemmata for $p(x,\xl)$ are nearly identical to the corresponding ones for $p(x,\mu)$, and indeed often easier since $\xl^2(n) = 1$ for all $n$, so our analyses in this section are notably briefer than those of the previous sections.

While the sets $\fM(X,Q)$, $\fm(X,Q)$, $\fP$, and $\fP_*$ are all defined as in section \ref{sec:prelim}, our choice of $Q$ for $p(x,\xl)$ is not the same as it is for $p(x,\mu)$, so we start anew with an unknown $Q$ such that $1 \leq Q \leq X$. All other notations are also maintained, but we emphasize that $\Phi(z)$ and $\Psi(z)$ now denote $\Phi(z,\xl)$ and $\Psi(z,\xl)$, respectively. In particular, equations \eqref{eq:sdp} and \eqref{eq:sdp*} are still written as  $\PsiD{1}(\rho) = x$ and $\PsiD{1}(-\rho)=x$. As done in our analysis of $\Psi(\xr e(\xa),\mu)$, in analyzing $\Psi(\xr e(\xa),\xl)$ we do not assume $\xr$ to be a solution to \eqref{eq:sdp} or \eqref{eq:sdp*} until necessary.

Finally we note that the functions $F_0(z)$ and $F_1(z)$ of \eqref{F:eq:F0F1Defin} are now given by the formulae
    \[
        F_0(z) = \sum_{n=1}^\infty z^n \qquad \text{and} \qquad F_1(z) = \sum_{n=1}^\infty \xl(n) z^n.
    \]

For the analysis of $\Psi(\rho e(\xa))$ for $\xa \in \fm(X,Q)$, Davenport's inequality \eqref{eq:Davenport} for $S_\mu(t,\xq)$ is replaced by a similar bound for $S_\xl(t,\xq)$ due to Bateman and Chowla.

\begin{lemma}[\cite{bateman1963some}*{Lem.\ 1}]
	\label{L:lem:BCBB}
    Fix $A > 0$. For $t > t_0$ one has $S_\xl(t,\xq) \less t (\log t)^{-A}$ uniformly for $\xq \in \rr$.
\end{lemma}

\begin{proof}
    By the well known identity $\xl(n) = \sum_{d^2 \;|\; n} \mu(n/d^2)$, one has
        \[
            \sum_{n \leq t} \xl(n) e(n\xq) = \sum_{d \leq \rt{t}} \, \sum_{m \leq t/d^2} \mu(m) e(md^2\xq) = \sum_{d \leq \rt{t}} S_{\mu}(t d^{-2}, d^2 \xq).
        \]
    Using Davenport's inequality \eqref{eq:Davenport} when $d \leq t^{1/4}$, and trivially bounding
        \[ 
            \lf| S_{\mu}(t d^{-2}, d^2 \xq) \rh| \leq \sum_{t^{1/4} < d \leq \rt{t}} t d^{-2} \leq t \sum_{d > t^{1/4}} d^{-2} \less t \cdot t^{-1/4},
        \]
    the result follows.
\end{proof}

Because Lemma \ref{L:lem:BCBB} provides a perfect substitute for Davenport's inequality \eqref{eq:Davenport}, we simply repeat the proof of Lemma \ref{M:lem:Minor1} to immediately recover the estimate 
    \<
        \label{L:eq:MajorPsi1BB}
        \Psi_1(\rho e(\xa)) \less X (\logX)^{-A} \qquad (\xa \in \rr).
    \>
Our analysis of $\Psi_0(\rho e(\xa))$ is greatly expedited by the well known inequality
    \<
        \label{eq:GeomSumBB}
        S_{\xl^2}(t,\xq) = \sum_{n\leq t} e(n\xq) \less \min\{t,\|\xq\|^{-1}\},
    \>
and we remark that \eqref{eq:GeomSumBB} serves as the $\Psi(z,\xl)$-analogue of Lemma \ref{lem:Brudern}.

\begin{lemma}
    \label{L:lem:minor}
	Fix $A > 0$ and let $Q = X^{1/3}$. For $X > X_0(A)$ and $\xa \in \fm(X,Q)$ one has
		\[
			\Psi(\rho e(\xa)) \less X (\logX)^{-A}.
        \]
\end{lemma}

\begin{proof}
	By inequality \eqref{L:eq:MajorPsi1BB} it suffices to show that $\Psi_0(\rho e(\xa)) \less X(\logX)^{-A}$. 
    
    For the sake of exposition we begin with parameters $\xd, \xk \in (0,1/2)$ and let $Q=X^\xd$, let $K=Q^\xk$, and let $k \leq K$. By Lemma \ref{lem:MinorIncl} we have $\{k\xa\} \in \fm(X,Q^{1-\xk})$ for all $\xa \in \fm(X,Q)$ and all $k \leq Q^{\xk}$, which implies (for said $\xa$ and $k$) that $\|k\xa\| > Q^{1-\xk}/X$, i.e., that
		\<  
            \label{L:eq:kalphaNormInvBB}
			\|k\xa\|^{-1} < X Q^{\xk-1}.
		\>
	As $F_0(z) = \sum_{n=1}^\infty z^n$, by partial summation we have 
		\[
			F_0(\rho^k e(k\xa)) = \sum_{n = 1}^\infty e^{nk/X} e(k\xa) = \frac{k}{X} \int_0^\infty e^{-tk/X} S_{\xl^2}(t, k\xa) \,dt.
      	\]
	Bounding $S_{\xl^2}(t,k\xa) \less t$ when $t \leq XQ^{\xk-1}$ and $S_{\xl^2}(t,k\xa) \less \|k\xa\|^{-1}$ otherwise, it follows that
		\begin{align}
            \label{L:eq:F0Int2}
			F_0(\xr^{k} e(k\xa)) \less \frac{X}{k} \int_0^{kQ^{\xk-1}} e^{-u}u \,du + \|k\xa\|^{-1} \int_{k Q^{\xk-1}}^\infty e^{-u} \,du.
        \end{align}
    Integrating by parts and applying \eqref{L:eq:kalphaNormInvBB}, the right side of \eqref{L:eq:F0Int2} is seen to be
        \[			
            \less \frac{X}{k} \big( 1-(1+kQ^{\xk-1})e^{-kQ^{\xk-1}} \big) + XQ^{\xk-1} e^{-kQ^{\xk-1}} = \frac{X}{k}(1-e^{-kQ^{\xk-1}}).
		\]
	As $k Q^{\xk-1} \leq Q^{2\xk-1}$ and $2\xk-1 < 0$, we bound $1-e^{-u} \leq u$ for $0 \leq u \leq 1$ to deduce that
		\[
			F_0(\rho^k e(k\xa)) \less \frac{X}{k} ( 1 - e^{-X^{\xd(2\xk-1)}} ) \leq \frac{X}{k}X^{\xd(2\xk-1)}.
        \]
    Using $K = Q^\xk = X^{\xd\xk}$ we apply Lemma \ref{F:lem:Fundamental} to conclude that 
		\[
			\Psi_0(\rho e(\xa)) \less X^{1 + \xd(2 \xk - 1)} + X^{1 - \xd \xk},
		\]
	and letting $\xd = \xk = \tf13$ yields the bound $\Psi_0(\xr e(\xa)) \less X^{\fr{8}{9}}$, which is more than needed.
\end{proof}

We now consider $\Psi(\rho e(\xa))$ with $\xa$ in $\fM(X,Q)$, recalling the definition
    \[
        \fz = \xz(2) = \pi^2/6.
    \]

\begin{lemma}
    \label{L:lem:Major}
    Fix $A>0$, let $X > X_0(A)$, and let $Q = X^{1/3}$. Suppose that $\xa = a/q + \xb$ with $0 \leq a \leq q \leq Q$, $(a,q)=1$, and $|\xb| \leq Q/qX$. One has
    	\[
    		\Psi(\rho e(\xa)) =  \Wmaj(q) \pfr{\fz X}{1-2\pi iX\xb} + O(X(\logX)^{-A}),
    	\]
    where
        \[
            \Wmaj(q) = \begin{cases} \tf14q^{-2} & 2 \nmid q, \\ q^{-2} & 2 \mid q. \end{cases}
        \]
\end{lemma}

\begin{proof}
    As $\Psi_1(\xr e(\xa)) \less X (\log X)^{-A}$ per \eqref{L:eq:MajorPsi1BB}, it remains only to analyze $\Psi_0(\rho e(\xa))$. Again let $\tau = X^{-1}(1-2\pi iX\xb)$ so that $\rho e(\xa) = e(a/q)e^{-\tau}$, and set $q_k := q/(q,k)$ and $a_k := ak/(q,k)$. Grouping summands modulo $q_k$, we have
    	\<
    		\label{L:eq:Major2}
            F_0( \rho^k e(k\xa) )  = \sum_{n=1}^\infty e\pfr{nka}{q} \rho^{nk} e(nk\xb)
    		= \sum_{\ell =1}^{q_k} e\pfr{a_k \ell}{q_k} \sum_{ n \equiv \ell \mod{q_k}} e^{-nk \tau}.
    	\>

	Summing by parts we find that $\sum_{n\equiv\ell \mod{q_k}} e^{-nk\tau}$ is equal to
        \begin{align*}
            k \tau \int_0^\infty e^{-xk \tau} \lf(\frac{x}{q_k} + O(1)\rh) \,dx = \fr1{q_k}\int_0^\infty e^{-xk\tau} xk\tau \,dx + \O{k|\tau|\int_0^{\infty}e^{-xk/X} \, dx},
        \end{align*}
    and arguing as in the derivation of \eqref{M:eq:MajorGammaInt} we deduce that
        \<
            \label{L:eq:MajorMainInt}
            \sum_{n\equiv\ell \mod{q_k}} e^{-nk\tau} = \frac{1}{kq_k\tau} + O(|1-2\pi iX\xb|) = \frac{1}{kq_k\tau} + O\pfr{Q}{q}.
        \>
	Combining \eqref{L:eq:Major2} and \eqref{L:eq:MajorMainInt}, it follows that
        \[
            F_0( \rho^k e(k\xa) ) 
            = \frac{1}{kq_k\tau} \sum_{\ell =1}^{q_k} e\pfr{\ell a_k }{q_k}  
            + O(Q).
        \]
    Since $(a_k,q_k) = 1$, the sum here is equal to $q_k$ when $q_k=1$ (i.e., when $q\;|\;k$) and equal to 0 otherwise, whereby
        \[
            \sum_{\substack{k \leq K \\ 2 \;\mid\; k}} \frac{1}{k} F_0(\rho^k e(k\xa)) 
			= \xt^{-1} \sum_{\substack{ k \leq K \\ 2\;\mid\;k \\ q\;\mid\;k }} \fr{1}{k^2} + \O{Q \log K}.
        \]
    Conditioning on the parity of $q$ and noting that $|\tau^{-1}| \leq X$, we find that
        \[
            \sum_{\substack{k \leq K \\ 2 \;\mid\; k}} \frac{1}{k} F_0(\rho^k e(k\xa))
            = \fz \Wmaj(q) \xt^{-1} + O\pfr{X}{K} + O(Q\log K),
        \]
    and by applying Lemma \ref{F:lem:TailBB} it follows that
        \[
            \Psi_0(\rho e(\xa)) = \fz \Wmaj(q) \tau^{-1} + O(X/K) + \O{Q\log K}.
        \]
    The result then follows upon substituting $Q = X^{1/3}$ and $K = Q^{1/3} = X^{1/9}$, in agreement with the parameters $\xd$ and $\xk$ used in proving Lemma \ref{L:lem:minor}.
\end{proof}

\begin{remark}
    \label{L:rem:VW}
    Recalling from Proposition \ref{M:prop:Major} the constants $V(q)$, which are given by the formula
        \[
            V(q) = \begin{cases}
                (2q)^{-2} \tilde{V}(q) & 2 \nmid q, \\
                q^{-2} \tilde{V}(q/2) & 2 \mid q,
            \end{cases}
        \]
    where $\tilde{V}(q)$ in the multiplicative function defined on prime powers $p^k>1$ via
        \[
            \tilde{V}(p^k) = \begin{cases}
                0 & k=1, \\
                -p^2 & k \geq 2,
            \end{cases}
        \]
    we see that $V(q)$ and $W(q)$ have structures as functions of $q$. In this spirit, one could, pedantically, even say that
        \[
            W(q) = \begin{cases}
                (2q)^{-2} \tilde{W}(q) & 2 \nmid q, \\
                q^{-2} \tilde{W}(q/2) & 2 \mid q,
            \end{cases} \qquad\text{where}\qquad \tilde{W}(q) := 1.
        \]
\end{remark}

As done for our analysis of $\Psi(\xr e(\xa),\mu)$ for $\xa$ in the principal arcs, we again let
    \[
        \xh = X^{-1}(\logX)^{-\nfr14} \qqand \xh_* = X_*^{-1}(\logX_*)^{-\nfr14}
    \]
and identify the sets
    \[
        \{e(\xa):\xa\in\fP\} \qqand \{e(\xa):\xa\in\fP_*\}
    \]
with
    \[
        \{ e(\xb) : \text{$\xb \in \rr$ and $|\xb|\leq \xh$} \} \qqand \{ -e(\xb) : \text{$\xb\in\rr$ and $|\xb|\leq\xh_*$} \},
    \]
respectively. As such we have the following analogue of Proposition \ref{M:prop:Princ}.

\begin{lemma}
    \label{L:lem:Princ}
    Fix $A > 0$. For all $X > X_0(A)$, $j \geq 0$, and $|\xb|\leq\xh$, one has
        \begin{align}
            \label{eq:xlprc}
		    \PsiD{j}(\pm\rho e(\xb)) &= j!\: \fr{\fz}{4} \pfr{X}{1-2\pi iX\xb}^{j+1}  + O_j(X^{j+1}(\logX)^{-A}), \\
            \label{eq:xlprc2}
            (\pm1)^j\Psid{j}(\pm\rho e(\xb)) &= j!\: \fr{\fz}{4} \pfr{X}{1-2\pi iX\xb}^{j+1} + O_j(X^{j+1}(\logX)^{-A}).
        \end{align}
\end{lemma}

\begin{proof}
	The proof is nearly identical to the proof of Proposition \ref{M:prop:Princ}, following the requisite substitutions. In particular, the integrands now involve
		\begin{align*}
			\Dir{\xl}{s} &= \fr{\xz(2s)}{\xz(s)} \qqand \Dir{\xl^2}{s} = \xz(s)
		\end{align*}
	in place of $\Dir{\mu}{s}$ and $\Dir{\mu^2}{s}$, respectively. Our main terms in \eqref{eq:xlprc} and \eqref{eq:xlprc2} now come from a residue of
        \[
            2^{-2}\xz(2)\xG(j+1)\xt^{-j-1} = j!\: \fr{\fz}{4} \pfr{X}{1-2\pi iX\xb}^{j+1},  
        \] 
    cf.\ expression \eqref{eq:mures}. When $z = -\rho e(\xb)$, similar to formula \eqref{M:eq:DirNeg1} we find that
        \[
            \Dir{(-1)^n \xl(n)}{s} = -(1+2^{1-s})\frac{\xz(2s)}{\xz(s)},
        \]
    and again the arguments used in proving Proposition \ref{M:prop:Princ} yield the result.
\end{proof}

Using Lemma \ref{L:lem:Major} in place of Proposition \ref{M:prop:Major}, the arguments of section \ref{sec:Relations} for Lemma \ref{lem:sdpt} and Theorem \ref{M:thm:Relations} similarly yield the uniqueness of the solutions $\xr$ and $\xr_*$ of \eqref{eq:sdp} and \eqref{eq:sdp*}, respectively, as well as the relations of Theorem \ref{L:thm:Relations}. For convenience, again using $\fz = \xz(2)$ we record here the particular relations
    \begin{alignat}{2}
        \label{L:eq:XinxRecall}
        -x \log \rho &= \tf12 \rt{\fz x}\brk{1+O((\logx)^{-A})},
            \qquad X &&= 2\rt{x/\fz}\brk{1+O((\logx)^{-A})}, \\
        \label{L:eq:X*inxRecall}
        -x \log \rho_* &= \tf12 \rt{\fz x}\brk{1+O((\logx)^{-A})},
            \qquad X_* &&= 2\rt{x/\fz}\brk{1+O((\logx)^{-A})}.
    \end{alignat}

As done in section \ref{sec:transference} we again use an ``arc transference'' to write
    \<
        \label{eq:pxxlNewInt}
        p(x,\xl) = \dint_{\rho e(\fA)} \Phi(z) z^{-x-1}\,dz + \dint_{\rho_* e(\fA_*)} \Phi(z) z^{-x-1} \, dz + \fT(\rho,\rho_*),
    \>
where again
    \[
        \fA := [0,\tf14)\cup[\tf34,1), \qquad \fA_* := [\tf14,\tf34),   
    \]
and
    \[
        \fT(\xr,\xr_*) := \dint_{i\xr}^{i\xr_*} \Phi(z)z^{-n-1} \,dz + \dint_{-i\xr_*}^{-i\xr} \Phi(z)z^{-n-1} \,dz.
    \]
Once again, the term $\fT(\xr,\xr_*)$ only contributes a small error term to formula \eqref{eq:pxxlNewInt}.

\begin{lemma}[cf.\ Lem.\ \ref{lem:Transfer}]
    \label{L:lem:Transfer}
    For $x > x_0$ let $\rho$ and $\rho_*$ be the unique solutions to equations \eqref{eq:sdp} and \eqref{eq:sdp*}, respectively. For all $\xe > 0$, as $x\to\infty$ one has
        \[
            \fT(\rho,\rho_*) \less \expp{(\tf34+\xe)\sqrt{\fz x}}.
        \]
\end{lemma}

\begin{proof}
    Fixing $A,\xe > 0$ and letting $x > x_0(A)$, relations \eqref{L:eq:XinxRecall} and \eqref{L:eq:X*inxRecall} show that
        \<
            \label{L:eq:rhoInt}
            \int_{\rho}^{\rho_*} \frac{dt}{t^{x+1}}
            = \fr{1}{x}\lf[ e^{\fr12\sqrt{\fz x}+o(\sqrt{x})} - e^{\fr12\sqrt{\fz x}+o(\sqrt{x})} \rh]
            \less e^{(\fr12+\xe)\sqrt{\fz x}}.
        \>
    For $\rho \leq t \leq \rho_*$ define $X_t$ via $t = \exp(-1/X_t)$, and assume $x$ (and $\rho$) to be large enough that $X_t (\logX_t)^{-A}$ is increasing as a function of $t$ for $t \geq \rho$. As $W(4)=\tf{1}{16}$ in Lemma \ref{L:lem:Major}, for $\rho \leq t \leq \rho_*$ and sufficiently large $\xr$, $X_*$, and $x$, we have
        \<
            \label{L:eq:RePsit}
            \Re \Psi(it) 
            = \fr{\fz}{16} X_t + O(X_t(\logX_t)^A) 
            \leq \fr{\fz}{16} X_* + O(X_*(\logX_*)^{-A}) \leq \fr{\fz}{8} X_*.
        \>
    Using equation \eqref{L:eq:X*inxRecall}, we further have
        \<
            \label{L:eq:RePsitBB}
            \fr{\fz}{8} X_* = \tf14 \rt{\fz x} + O(\rt{x}(\logx)^{-A}) \leq (\tf14+\xe)\sqrt{\fz x},
        \>
    and the result then follows from \eqref{L:eq:rhoInt}--\eqref{L:eq:RePsitBB}, akin to the proof of Lemma \ref{lem:Transfer}.
\end{proof}

Due to the overwhelming similarities between Theorems \ref{M:thm:Relations} and \ref{L:thm:Relations} and between Lemmata \ref{lem:Transfer} and \ref{L:lem:Transfer}, and the parallels drawn in Remark \ref{L:rem:VW}, the remaining arguments of sections \ref{sec:transference}--\ref{sec:MAsymp} are naturally ported to analogous arguments concerning $p(x,\xl)$, $\Phi(z,\xl)$, etc.. As such, there is little to be gained by repeating said arguments, so we merely state the salient results concerning $p(x,\xl)$ and reference their respective $p(x,\mu)$-analogues.

\begin{lemma}[cf.\ Lem.\ \ref{M:lem:NonPrincIneq} and Prop.\ \ref{M:prop:pxmuReduc}]
    \label{L:lem:pxlReduc}
    For $x > x_0$ let $\xr$ and $\xr_*$ be the unique solutions to equations \eqref{eq:sdp} and \eqref{eq:sdp*}, respectively. For fixed $B>0$, as $x\to\infty$ one has
        \[
            \begin{aligned}
                p(x,\xl) &= \rho^{-x} \int_{\fP} \Phi(\rho e(\xa)) e(-x\xa) \,d\xa + \rho_*^{-x} \int_{\fP_*} \Phi(\rho e(\xa)) e(-x\xa) \,d\xa \\
                & \qquad + O(\rho^{-x}\Phi(\rho)x^{-B}) + (\rho_*^{-x}\Phi(-\rho_*)x^{-B}) + O(e^{(\fr34+\xe)\sqrt{\fz x}}).
            \end{aligned}
        \]
\end{lemma}

\begin{lemma}[cf.\ Prop.\ \ref{M:prop:MainTerm} and eq.\ \eqref{M:eq:MainTerm*}]
    \label{L:lem:PrincInts}
    For all $x > x_0$ let $\rho$ and $\rho_*$ be the unique solutions to equations \eqref{eq:sdp} and \eqref{eq:sdp*}, respectively. As $x \to \infty$, one has
        \begin{align*}
            \rho^{-x} \int_{\fP} \Phi(\rho e(\xa)) e(-x\xa) \,d\xa 
                &= \frac{\rho^{-x}\Phi(\rho)}{\sqrt{2\pi \PsiD{2}(\rho)}}\lf[ 1 + O(x^{-\nfr15}) \rh], \\
            \rho_*^{-x} \int_{\fP_*} \Phi(\rho_* e(\xa)) e(-x\xa) \,d\xa 
                &= (-1)^{-x} \frac{\rho_*^{-x}\Phi(-\rho_*)}{\sqrt{2\pi \PsiD{2}(-\rho_*)}} \lf[ 1 + O(x^{-\nfr15}) \rh].
        \end{align*}
    where $(-1)^{-x} = \exp(-\pi ix)$.
\end{lemma}

Combining Lemmata \ref{L:lem:pxlReduc} and \ref{L:lem:PrincInts} and repeating the arguments used in deriving equation \eqref{M:eq:pxmuAsymp2} from equation \eqref{M:eq:pxmuAsymp}, Theorem \ref{L:thm:pnlAsymp} is established.

\bibliography{spnmu.bib}
\end{document}